\documentclass[10pt,reqno,oneside,a4paper]{amsart}

\usepackage{subfiles}
\usepackage[english]{babel}
\usepackage[utf8]{inputenc}
\usepackage{amsmath}
\usepackage{amsthm}
\usepackage{amssymb}
\usepackage{amsfonts}
\usepackage{graphicx}
\usepackage[shortlabels]{enumitem}
\usepackage{geometry}
\usepackage{adjustbox}
\usepackage{listings}
\usepackage{color}
\usepackage{upgreek}
\usepackage{mathtools}
\usepackage{color}
\usepackage{empheq}
\usepackage{bm}
\usepackage{csquotes}
\usepackage{tikz}
\usetikzlibrary{matrix}
\usepackage{tikz-cd}
\usepackage{setspace}
\usepackage{commath}
\usepackage[norelsize]{algorithm2e}
\usepackage{adjustbox}
\usepackage{mathrsfs}
\usepackage{stmaryrd}
\usetikzlibrary{decorations.pathmorphing,shapes}
\usepackage{rotating}
\usepackage{pdflscape}
\usepackage{subcaption}
\usepackage{afterpage}
\usepackage[all,cmtip]{xy}
\usepackage{tabu}
\usepackage{stackengine}
\usepackage{caption}
\usetikzlibrary{arrows,calc}
\usepackage{hyperref}
\usepackage{scalerel}
\usepackage{centernot}
\usepackage{float}

\allowdisplaybreaks

\newcommand{\tikzAngleOfLine}{\tikz@AngleOfLine}
\def\tikz@AngleOfLine(#1)(#2)#3{%
\pgfmathanglebetweenpoints{%
\pgfpointanchor{#1}{center}}{%
\pgfpointanchor{#2}{center}}
\pgfmathsetmacro{#3}{\pgfmathresult}%
}

\newcommand{\proj}{\text{\normalfont proj-}}

\newcommand{\MCM}{\text{\normalfont MCM-}}
\newcommand{\Mod}{\text{\normalfont Mod-}}

\newcommand{\gap}{\hspace{1pt}}
\newcommand{\ZZ}{\mathbb{Z}}
\newcommand{\CC}{\mathbb{C}}

\renewcommand{\mod}{\text{\normalfont mod-}}
\newcommand{\stabmod}{\text{\normalfont \underline{mod}-}}
\newcommand{\stabMCM}{\text{\normalfont \underline{MCM}-}}

\DeclareMathOperator{\GKdim}{GKdim}

\DeclareMathOperator{\im}{im}
\DeclareMathOperator{\coker}{coker}
\DeclareMathOperator{\gldim}{gl.\hspace{-1pt}dim}
\DeclareMathOperator{\pdim}{p.\hspace{-1pt}dim}
\DeclareMathOperator{\idim}{i.\hspace{-1pt}dim}

\DeclareMathOperator{\End}{End}
\DeclareMathOperator{\stabEnd}{\underline{End}}
\DeclareMathOperator{\Spec}{Spec}

\DeclareMathOperator{\Hom}{Hom}

\DeclareMathOperator{\Dsg}{\mathcal{D}_{\normalfont{\text{sg}}}}
\DeclareMathOperator{\stabHom}{\underline{Hom}}
\DeclareMathOperator{\Ext}{Ext}

\DeclareMathOperator{\grr}{gr}

\DeclareMathOperator{\add}{add}
\DeclareMathOperator{\real}{Re}
\DeclareMathOperator{\imaginary}{Im}

\DeclareMathOperator{\sspan}{span}

\newcommand{\id}{\text{\normalfont id}}

\DeclarePairedDelimiter\ceil{\lceil}{\rceil}
\DeclarePairedDelimiter\floor{\lfloor}{\rfloor}

\numberwithin{equation}{section}

\theoremstyle{definition}
\newtheorem{defn}[equation]{Definition}

\theoremstyle{plain}
\newtheorem{thm}[equation]{Theorem}
\newtheorem{prop}[equation]{Proposition}
\newtheorem{lem}[equation]{Lemma}
\newtheorem{cor}[equation]{Corollary}

\newtheorem*{thm*}{Theorem}
\newtheorem{ass}[equation]{Assumption}

\theoremstyle{remark}
\newtheorem{rem}[equation]{Remark}

\newtheorem{example}[equation]{Example}

\newcommand\restr[2]{{
  \left.\kern-\nulldelimiterspace 
  #1 
  \right|_{#2} 
  }}

\newcounter{sarrow}

\newcounter{darrow}

\makeatletter
\renewcommand*\env@matrix[1][\arraystretch]{%
  \edef\arraystretch{#1}%
  \hskip -\arraycolsep
  \let\@ifnextchar\new@ifnextchar
  \array{*\c@MaxMatrixCols c}}
\makeatother

\newboolean{printBibInSubfiles}
\setboolean{printBibInSubfiles}{true}
\def\bib{\ifthenelse{\boolean{printBibInSubfiles}}
           { \bibliographystyle{amsalpha} \bibliography{thesisbib} }
       {}
}

\makeatletter
\tikzset{
  column sep/.code=\def\pgfmatrixcolumnsep{\pgf@matrix@xscale*(#1)},
  row sep/.code   =\def\pgfmatrixrowsep{\pgf@matrix@yscale*(#1)},
  matrix xscale/.code=%
    \pgfmathsetmacro\pgf@matrix@xscale{\pgf@matrix@xscale*(#1)},
  matrix yscale/.code=%
    \pgfmathsetmacro\pgf@matrix@yscale{\pgf@matrix@yscale*(#1)},
  matrix scale/.style={/tikz/matrix xscale={#1},/tikz/matrix yscale={#1}}}
\def\pgf@matrix@xscale{1}
\def\pgf@matrix@yscale{1}
\makeatother

\definecolor{mygray}{gray}{0.8}

\title{Singularity categories of deformations of Kleinian singularities}
\author{Simon Crawford}
\address{Department of Pure Mathematics, University of Waterloo, 200 University Ave W, Waterloo, ON N2L 3G1, Canada}
\email{simon.crawford@uwaterloo.ca}
\date{\today}


\begin{document}

\begin{abstract}
Let $G$ be a finite subgroup of $\text{SL}(2,\Bbbk)$ and let $R = \Bbbk[x,y]^G$ be the coordinate ring of the corresponding Kleinian singularity. In 1998, Crawley-Boevey and Holland defined deformations $\mathcal{O}^\lambda$ of $R$ parametrised by weights $\lambda$. In this paper, we determine the singularity categories $\Dsg(\mathcal{O}^\lambda)$ of these deformations, and show that they correspond to subgraphs of the Dynkin graph associated to $R$. This generalises known results on the structure of $\Dsg(R)$. We also provide a generalisation of the intersection theory appearing in the geometric McKay correspondence to a noncommutative setting.
\end{abstract}

\maketitle

\section{Introduction}

\subsection{Background}
Throughout let $\Bbbk$ be an algebraically closed field of characteristic $0$. The Kleinian singularities $\Bbbk^2/G$, where $G$ is a finite subgroup of $\text{SL}(2,\Bbbk)$, are ubiquitous in algebraic geometry, representation theory, and singularity theory. In this paper, we shall study the latter of these for a family of (generically noncommutative) algebras. \\
\indent The notion of the singularity category of a ring $R$ was introduced by Buchweitz in \cite{buch86} as a particular Verdier quotient of $\mathcal{D}^{\text{b}}(\mod R)$. More specifically, writing $\text{Perf}(R)$ for the full subcategory of perfect complexes in $\mathcal{D}^{\text{b}}(\mod R)$, Buchweitz defined the singularity category as the Verdier quotient category
\begin{gather*}
\Dsg (R) \coloneqq \frac{\mathcal{D}^{\text{b}}(\mod R)}{\text{Perf}(R)}.
\end{gather*}
By construction, this category possesses the structure of a triangulated category. Buchweitz also showed that, when $R$ is Gorenstein, the singularity category is triangle equivalent to $\stabMCM R$, the stable category of maximal Cohen-Macaulay $R$-modules (that this latter category is triangulated also follows from a general result of Happel, see \cite{happel}). The singularity category of a commutative ring $R$ is also closely related to the category of reduced matrix factorisations of $R$ by \cite[Corollary 6.3]{eisenbud}, and under mild hypotheses these categories are in fact equivalent. \\
\indent From the above definition, it is not difficult to see that $\Dsg(R)$ is trivial precisely when $R$ has finite global dimension. However, in general it is difficult to give an adequate description of the singularity category of an arbitrary Gorenstein ring of infinite global dimension. Recent work includes \cite{chen11,chen15} which describes the singularity category when $R$ has radical square zero or when it is a quadratic monomial algebra, and \cite{kalck} which provides a description when $R$ is a so-called gentle algebra. Moreover, in  \cite{ari} the authors determine the singularity categories of some commutative Gorenstein isolated singularities. \\
\indent The standard examples of commutative surface singularities are the Kleinian singularities, which are very well understood. The main aim of this paper is to provide a concrete description of the singularity category of certain noncommutative deformations of the coordinate ring of a Kleinian singularity. \\
\indent Very little is known about the singularities of noncommutative rings, particularly those which are not finite over their centre. For example, given a singular noncommutative ring $S$, it is not known whether and under what circumstances one can find commutative rings $R_1, \dots, R_k$ such that we have an equivalence of triangulated categories $\Dsg(S) \simeq \bigoplus_{i=1}^k \Dsg(R_i)$. If this is the case, one can think of $S$ as having the same singularities as those of the varieties $\Spec R_i$. Our main result, Theorem \ref{introthm0}, shows that we have such a decomposition of singularity categories for the algebras of interest in this paper, and can be seen as a first step towards better understanding singularities of noncommutative surfaces. \\
\indent In \cite{cbh}, Crawley-Boevey and Holland defined a family of algebras $\mathcal{O}^\lambda(\widetilde{Q})$ depending on the data of an extended Dynkin quiver $\widetilde{Q}$ and a so-called weight for $\widetilde{Q}$. Write $Q$ for the Dynkin quiver obtained from $\widetilde{Q}$ by removing an extending vertex, and write $R_Q$ for the coordinate ring of the corresponding Kleinian singularity. Then the algebras $\mathcal{O}^\lambda (\widetilde{Q})$ may be thought of as deformations of $R_Q$ in the sense that there exists a filtration $\mathcal{F}$ of $\mathcal{O}^\lambda (\widetilde{Q})$ satisfying $\grr_{\mathcal{F}} \mathcal{O}^\lambda (\widetilde{Q}) \simeq R_Q$. These deformations are generically noncommutative, a property which depends on the weight $\lambda$, and it is easy to determine when this is the case. When $\widetilde{Q} = \widetilde{\mathbb{A}}_n$, if $\mathcal{O}^\lambda (\widetilde{Q})$ is noncommutative then it is an example of a generalised Weyl algebra, as studied in \cite{bavula,hodges}. If $\mathcal{O}^\lambda(\widetilde{Q})$ is commutative, a description of its singularity category follows from \cite[Theorem 3.2]{qfactorial}, where quite geometric techniques are employed. Through a completely ring-theoretic approach, we determine $\Dsg(\mathcal{O}^\lambda(\widetilde{Q}))$ irrespective of whether the deformation is commutative or noncommutative. \\
\indent Our main result can be stated as follows, where undefined terms will be defined in Section \ref{prelim}.

\begin{thm}[Theorem \ref{addequiv}, Theorem \ref{triequiv}] \label{introthm0}
Let $\widetilde{Q}$ be an extended Dynkin quiver with vertex set $\{0,1, \dots , n\}$, where $0$ is an extending vertex, and write $Q$ for the full subquiver obtained by deleting vertex $0$. Let $\lambda$ be a weight for $\widetilde{Q}$. Then there exists a subset $J = J(\lambda)$ of $\{1, \dots , n\}$ such that, if $Q^{(1)} \sqcup \dots \sqcup Q^{(r)}$ is the full subquiver of $Q$ obtained by deleting the vertices in $J$, so that the $Q^{(i)}$ are connected and therefore necessarily Dynkin, there is a triangle equivalence
\begin{align*}
\Dsg( \mathcal{O}^\lambda(\widetilde{Q})) \simeq \bigoplus_{i=1}^{r} \Dsg (R_{Q^{(i)}}).
\end{align*}
\end{thm}

One can show that we may restrict our attention to the case where the weight $\lambda$ is \emph{quasi-dominant} (see Definition \ref{quasidomdef}). For example, when $\Bbbk = \CC$, a weight is quasi-dominant if, for $1 \leqslant i \leqslant n$,
\begin{align*}
\lambda_i \in \{ z \in \mathbb{C} \mid \real z > 0, \text{ or } \real z = 0 \text{ and } \imaginary z \geqslant 0 \},
\end{align*}
and where $\lambda_0$ can be arbitrary. In this case, the subset $J$ in the above theorem is $J = \{ i \in \{1, \dots , n\} \mid \lambda_i = 0\}$. \\
\indent The result in Theorem \ref{introthm0} coincides with the intuition coming from commutative singularity theory which says that deforming a singularity should make it no worse; in our context, deforming a singularity corresponds to making weights at certain vertices of $\widetilde{Q}$ nonzero, and the above theorem says that this makes the singularity category simpler, in a precise sense. In the appendix, we also provide a simple proof which illustrates how the translation functor acts on the triangulated category $\Dsg(\mathcal{O}^\lambda(\widetilde{Q}))$. \\ 
\indent Since the first version of this paper appeared online, an alternative proof of Theorem \ref{introthm0} has been given in \cite[Theorem 9.4]{kalck2018} using relative singularity categories. \\
\indent Now suppose that the weight $\lambda \in \Bbbk^{n+1}$ is given by $\lambda_0 = 1$ and $\lambda_i = 0$ for $1 \leqslant i \leqslant n$, and in this case write $\lambda = \varepsilon_0$. We then consider $\mathcal{O}^\lambda(\widetilde{Q})$ to be a noncommutative analogue of $R_Q$. This viewpoint is partially justified by the following immediate corollary:

\begin{cor}
Retain the notation of Theorem \ref{introthm0}, and suppose that $\lambda = \varepsilon_0$, as above. Then there is a triangle equivalence
\begin{align*}
\Dsg( \mathcal{O}^\lambda(\widetilde{Q})) \simeq \Dsg(R_{Q}).
\end{align*}
\end{cor}

\indent Another family of results concerning Kleinian singularities is what is often called the geometric McKay correspondence, which concerns the intersection theory of the minimal resolution of $\Spec R_Q$. In the last section, we prove a result which may be seen as a generalisation of this to a noncommutative setting. We give an imprecise statement of this result below, and a more precise statement in Section \ref{inttheorysec}.

\begin{thm}[Theorem \ref{inttheory}]
Let $\widetilde{Q}$ be an extended Dynkin quiver with corresponding Dynkin quiver $Q$ and let $\lambda = \varepsilon_0$. Then $\mathcal{O}^\lambda(\widetilde{Q})$ has a noncommutative resolution, and the intersection theory of the exceptional objects in this resolution is the same as that of the exceptional curves in the minimal resolution of a Kleinian singularity of type corresponding to $Q$.
\end{thm}

In particular, this result further supports the viewpoint that $\mathcal{O}^\lambda(\widetilde{Q})$ can be viewed as a noncommutative analogue of $R_Q$ when $\lambda = \varepsilon_0$. \\
\indent We now take a moment to provide an overview of the proof of Theorem \ref{introthm0}. For all of our calculations, we work in $\stabMCM \mathcal{O}^\lambda(\widetilde{Q})$ rather than $\Dsg(\mathcal{O}^\lambda(\widetilde{Q}))$; as mentioned previously, these two categories are triangle equivalent, see Theorem \ref{buchthm}. The first important observation to make is that we can restrict our attention to weights $\lambda$ which are \emph{quasi-dominant}. In Section \ref{singcatsec}, this restriction allows us to give a concrete description of $\Dsg(\mathcal{O}^\lambda(\widetilde{Q}))$ as a $\Bbbk$-linear category in terms of an auxiliary Krull-Schmidt category. We also find that the isoclasses $V_i$ of indecomposable objects in $\Dsg(\mathcal{O}^\lambda(\widetilde{Q}))$ are indexed by those vertices $i \geqslant 1$ with $\lambda_i = 0$. This auxiliary category allows us to establish the equivalence of Theorem \ref{introthm0}, but only as a $\Bbbk$-linear equivalence. \\
\indent In a previous version of this paper, a lengthy case-by-case analysis was then used to show that we had the desired triangle equivalence. A result of Keller which subsequently appeared significantly reduces the amount of work that needs to be done. In fact, it now suffices to show that the direct sum appearing on the right hand side of the (a priori $\Bbbk$-linear) equivalence of Theorem \ref{introthm0} is a decomposition into algebraic triangulated subcategories. This is shown in Section \ref{SingChapterTriSec}. 
\subsection{Organisation of the paper}
This paper is organised as follows. In Section \ref{prelim}, we recall some basic definitions and facts, and introduce the notation used throughout the paper. In Section \ref{singcatsec}, the singularity categories of ${O}^\lambda(\widetilde{Q})$ are determined as $\Bbbk$-linear categories. We then complete the proof of Theorem \ref{introthm0} in Section \ref{SingChapterTriSec}. In Section \ref{inttheorysec}, we provide a noncommutative version of the geometric McKay correspondence, and in the appendix we detail how the translation functor behaves on objects of $\Dsg(R_Q)$.

\subsection{Acknowledgements}
The work contained in this paper was completed while the author was an EPSRC-funded student at the University of Edinburgh, and the material contained in this paper appears in an adapted form in his PhD thesis, \cite{simon}. The author would like to thank his supervisor Susan J. Sierra for suggesting this problem and providing guidance, Michael Wemyss for a number of useful discussions, and the anonymous referee whose suggestions ultimately led to a much shorter argument than that given in the first version of this paper. The author also thanks the EPSRC.


\section{Preliminaries} \label{prelim}
We now recall some of the definitions and results that we will make use of throughout this paper. In this section, $R$ will denote an arbitrary ring.

\subsection{Conventions}
As stated in the introduction, throughout $\Bbbk$ will denote an algebraically closed field of characteristic 0. We write $\mod R$ (respectively, $R\text{-mod}$) for the category of finitely generated right (respectively, left) $R$-modules; in this paper, we shall use right modules unless otherwise stated. We also write $\proj R$ for the full subcategory of $\mod R$ consisting of finitely generated projective modules. We write $M^* \coloneqq \Hom_R(M,R)$ for the dual of an $R$-module $M$, which is an $(R, \End_R(M))$-bimodule. We write $\pdim M$ and $\idim M$ for the projective and injective dimensions of $M \in \mod R$, respectively, and $\gldim R$ for the global dimension of $R$.

\subsection{Definitions and basic results}
\begin{defn} 
A \emph{quiver} $Q$ is a directed multigraph, and we write $Q_0$ for the set of vertices and $Q_1$ for the set of arrows. We equip $Q$ with head and tail maps $h,t : Q_1 \to Q_0$ which take an arrow to the vertices that are its head and tail respectively. A \emph{non-trivial path} in the quiver is a sequence of arrows $p = \alpha_1 \alpha_2 \dots \alpha_\ell$ with $h(\alpha_i) = t(\alpha_{i+1})$ for $1 \leqslant i \leqslant \ell - 1$ (that is, we compose arrows from left to right), and such a path is said to have \emph{length} $\ell$. Moreover, for each vertex $i \in Q_0$ there is a \emph{trivial path} $e_i$ of length $0$, with head and tail vertex both equal to $i$.
\end{defn}

\begin{defn}
Given a field $\Bbbk$ and a quiver $Q$, we define the path algebra $\Bbbk Q$ of $Q$ as follows: as a $\Bbbk$-vector space, $\Bbbk Q$ has a basis given by paths in the quiver, and we define multiplication by concatenation of paths:
\begin{gather*}
p \cdot q = \left\{ \begin{array}{cl}
pq & \text{if } h(p) = t(q), \\
0 & \text{otherwise.}
\end{array} \right.
\end{gather*}
\end{defn}

\indent If $R$ is a commutative ring, then $\Spec R$ is nonsingular if and only if $R$ has finite global dimension. It is therefore sensible to say that a (possibly noncommutative) ring is \emph{nonsingular} if it has finite global dimension, and \emph{singular} otherwise. Before we are able to define the singularity category of a ring, we must make a few more definitions.

\begin{defn}
Given $R$-modules $M$ and $N$, write $\stabHom_R(M,N) = \Hom_R(M,N)/\hspace{-3pt}\sim$, where $f \sim f'$ if and only if $f-f'$ factors through a finitely generated projective module. The \emph{stable module category} of $R$, denoted $\stabmod R$, is then the category whose objects are the same as those of $\mod R$, and for modules $M,N$, has morphisms $\stabHom_R(M,N)$. Given a full subcategory $\text{abc-}R$ of $\mod R$, we write $\text{\underline{abc}-}R$ for the full subcategory of $\stabmod R$ whose objects are the same as those of $\text{abc-}R$.
\end{defn}

\indent Noting that an element of $\sum_{i=1}^k n_i \otimes f_i$ of $N \otimes_R M^*$ gives rise to a homomorphism $M \to N$ via $m \mapsto \sum_{i=1}^k n_i  f_i(m)$, it is not hard to show that a module homomorphism $f: M \to N$ factors through a projective module if and only if $f$ is the image of some element of $N \otimes_R M^*$. Abusing notation, this allows us to identify $\stabHom_R(M,N)$ with $\Hom_R(M,N)/(N \otimes_R M^*)$, which will be useful in later calculations. In this paper, we are often in the situation where $R = e \Lambda e$, $M = e_1 \Lambda e$, and $N = e_2 \Lambda e$, where $\Lambda$ is some ring and $e, e_1$, and $e_2$ are pairwise orthogonal idempotents, and we are able to make identifications
\begin{align*}
M^* \cong e \Lambda e_1, \quad \Hom_R(M,N) \cong e_2 \Lambda e_1, \quad \stabHom_R(M,N) \cong \frac{e_2 \Lambda e_1}{e_2 \Lambda e \Lambda e_1}.
\end{align*}
\indent In the stable module category, we have a weaker notion of an isomorphism than in the usual module category. Indeed, \cite[Proposition 1.44]{ausbridge} shows that two $R$-modules $M,N$ are isomorphic in $\stabmod R$ if and only if there exist projective modules $P$ and $Q$ such that $M \oplus P \cong N \oplus Q$ in $\mod R$. \\
\indent The \emph{first syzygy} $\Omega M$ of $M \in \mod R$ is defined to be the kernel of any surjection $R^n \twoheadrightarrow M$. The observation in the previous paragraph combined with \cite[Proposition 8.5]{rot} implies that $\Omega M$ is uniquely determined in $\stabmod R$.

\begin{defn}
A ring $R$ is said to be \emph{Gorenstein} if it is noetherian (i.e., left and right noetherian) and both $\idim R_R$ and $\idim {}_R R$ are finite. By \cite[Lemma A]{zaks}, under these hypotheses the values $\idim R_R$ and $\idim {}_R R$ coincide, and we call this common value the \emph{(injective) dimension} of $R$.
\end{defn}
 
\begin{defn}
Suppose that $R$ is Gorenstein. A finitely generated $R$-module $M$ is said to be \emph{maximal Cohen-Macaulay} (MCM) if it satisfies $\Ext_R^i(M,R) = 0$ for all $i \geqslant 1$. We write $\MCM R$ for the full subcategory of $\mod R$ consisting of maximal Cohen-Macaulay $R$-modules.
\end{defn}

For commutative local rings, the above definition coincides with the usual (commutative) definition of maximal Cohen-Macaulay modules in terms of depth \cite[Section 4.2]{buch86}. Maximal Cohen-Macaulay modules have the following elementary properties, proofs of which can be found in \cite{buch86}:

\begin{lem}\leavevmode \label{MCMlemma}
\begin{enumerate}[{\normalfont (1)},leftmargin=*,topsep=0pt,itemsep=0pt]
\item Any finitely generated projective module is MCM.
\item MCM modules are reflexive.
\item Finite direct sums and direct summands of MCM modules are MCM.
\item An MCM module is either projective or has infinite projective dimension.
\end{enumerate}
\end{lem}

With these definitions in hand, we now recall a theorem which identifies a category that is triangle equivalent to the singularity category in the case of a Gorenstein ring $R$:

\begin{thm}[{\cite[Theorem 4.4.1]{buch86}}] \label{buchthm}
Suppose that $R$ is Gorenstein. Then the full subcategory $\stabMCM R$ of $\stabmod R$ whose objects are MCM $R$-modules is a triangulated category, with translation functor $\Sigma$ given by $\Sigma M = \Omega^{-1} M$. Moreover, there is a triangle equivalence $\Dsg (R) \simeq \stabMCM R$.
\end{thm}
 
While the term ``singularity category'' is more suggestive (which was our main reason for using this terminology in the introduction), since every example that we consider in this paper satisfies the hypotheses of this theorem, we instead focus our attention on determining $\stabMCM R$. \\
\indent Theorem \ref{buchthm} is a specific example of a more general result due to Happel, which we now briefly recall. An \emph{exact category} $\mathcal{C}$ is an additive category possessing a class of \emph{conflations} (sometimes called exact sequences) which are triples of objects connected by arrows $X \to Y \to Z$, and which satisfy a number of axioms; see \cite[Section 2]{frob} for more details. An object $P \in \mathcal{C}$ is \emph{projective} if the functor $\Hom_{\mathcal{C}}(P,-)$ sends conflations to exact sequences, and we say that $\mathcal{C}$ has \emph{enough projectives} if every object $Z \in \mathcal{C}$ fits into a conflation $X \to P \to Z$ with $P$ projective. Dually, one has a notion of an \emph{injective object} and of having \emph{enough injectives}. An exact category $\mathcal{C}$ is said to be \emph{Frobenius} provided that it has enough projectives and enough injectives, and the class of projective objects coincides with the class of injective objects. Given a Frobenius category $\mathcal{C}$, we may form its stable category $\underline{\mathcal{C}}$ in the same way we formed the stable category $\stabMCM R$. Then \cite{happel} shows that this category is triangulated, and if $X \to Y \to Z$ is a conflation in $\mathcal{C}$ then there exists a triangle of the form $X \to Y \to Z \to \Sigma X$ in $\underline{\mathcal{C}}$. If $\mathcal{T}$ is a triangulated category which is triangle equivalent to the stable category of a Frobenius category, then we say that $\mathcal{T}$ is \emph{algebraic}. \\
\indent If $R$ is a Gorenstein ring, then $\MCM R$ is Frobenius and so Happel's result implies that $\stabMCM R$ is triangulated; this triangulated structure is precisely the one given in Theorem \ref{buchthm}. In $\MCM R$, every conflation $X \to Y \to Z$ arises from a short exact sequence $0 \to X \to Y \to Z \to 0$ of MCM $R$-modules. \\
\indent Finally, we recall two useful results that will be helpful when identifying the maximal Cohen-Macaulay modules of a ring. Given an additive category $\mathcal{C}$ and an object $C \in \mathcal{C}$, we write $\add(C)$ for the full subcategory of $\mathcal{C}$ consisting of direct summands of finite direct sums of $C$. This is the smallest additive subcategory of $\mathcal{C}$ which contains $C$ and is closed under taking direct summands. The following result is due to Auslander, but we provide a proof.

\begin{prop}[Auslander] \label{auslander}
Suppose that $R$ is Gorenstein and that $M \in \MCM R$ is a generator (for example, this occurs if $M$ has $R$ as a direct summand or if $R$ is simple). If $\gldim \End_R(M) \leqslant 2$, then $\add M = \MCM R$.
\end{prop}
\begin{proof} 
Write $\Lambda = \End_R(M)$. Since $\mod R$ has split idempotents (a fact which holds for any ring $R$), \cite[Proposition 2.3]{krause} implies that the functor $\Hom_R(M,-): \mod R \to \mod \Lambda$ restricts to an equivalence
\begin{align}
\add M \xrightarrow{\simeq} \proj \Lambda. \label{addequivproj}
\end{align}
We also note that since $M$ is a generator, $R^n \in \add M$ for any $n \geqslant 1$. \\
\indent That $\add M \subseteq \MCM R$ is clear, so suppose that $N \in \MCM R$. Since $R$ is noetherian, $N^*$ is finitely presented, so we have an exact sequence of left $R$-modules of the form
\begin{align*}
R^{m} \to R^{n} \to N^* \to 0.
\end{align*}
Applying $\Hom_R(-,R)$ and noting that $N$ is MCM and therefore reflexive, we obtain an exact sequence
\begin{align*}
0 \to N \to R^{n} \to R^{m}.
\end{align*}
Applying $\Hom_R(M,-)$ then gives an exact sequence
\begin{align*}
0 \to \Hom_R(M,N) \to \Hom_R(M,R^{n}) \xrightarrow{\theta} \Hom_R(M,R^{ m}) \to \coker \theta \to 0,
\end{align*}
where, since $M$ is a generator, $\Hom_R(M,R^{n})$ and $\Hom_R(M,R^{ m})$ are both projective $\Lambda$-modules by (\ref{addequivproj}). Since $\gldim \Lambda \leqslant 2$ we have $\pdim \coker \theta \leqslant 2$, and therefore $\Hom_R(M,N)$ is also a projective $\Lambda$-module. By (\ref{addequivproj}), it follows that $N \in \add M$.
\end{proof}

\begin{rem}
If $R$ has injective dimension at most $2$ then the converse of Proposition \ref{auslander} also holds; see \cite[Proposition 2.2.11]{simon}.
\end{rem}

When $R$ has injective dimension at most 2, we also have the following:

\begin{lem} \label{MCMreflexive}
Let $R$ be a Gorenstein ring of injective dimension at most 2. Then $M \in \mod R$ is reflexive if and only if it is maximal Cohen-Macaulay.
\end{lem}
\begin{proof}
$(\Leftarrow)$ This is Lemma \ref{MCMlemma} (2), and doesn't require the hypothesis on injective dimension.\\
\indent $(\Rightarrow)$ Suppose now that $M$ is reflexive. Since $R$ is noetherian, $M^*$ is finitely presented, so we have an exact sequence of the form
\begin{gather*}
R^{m} \to R^{n} \to M^* \to 0.
\end{gather*}
Applying $\Hom_R(-,R)$ and noting that $M$ is reflexive yields an exact sequence
\begin{gather*}
0 \to M \to R^{n} \xrightarrow{\theta} R^{m} \to \coker \theta \to 0.
\end{gather*}
But then, by \cite[Corollary 6.55]{rot}, $\Ext_R^i(M,R) \cong \Ext_R^{i+2}(\coker \theta,R) = 0$ for all $i \geqslant 1$, where the last equality follows since $\idim R \leqslant 2$. That is, $M$ is maximal Cohen-Macaulay.
\end{proof}

\subsection{The deformations of Crawley-Boevey and Holland} In \cite{cbh}, Crawley-Boevey and Holland introduced the notion of the \emph{deformed preprojective algebra} of a quiver $Q$, and, if $\widetilde{Q}$ is extended Dynkin, a family of $\Bbbk$-algebras $\mathcal{O}^\lambda(\widetilde{Q})$ which may be thought of as deformations of the coordinate ring of a Kleinian singularity. We now recall these definitions, noting that our definition of $\mathcal{O}^\lambda(\widetilde{Q})$ differs slightly from that of Crawley-Boevey and Holland, but is consistent with their definition by \cite[Theorem 0.1]{cbh}.

\begin{defn}
Let $Q$ be a quiver without loops. The \emph{double} of $Q$ is the quiver $\overline{Q}$ obtained from $Q$ by adding a \emph{reverse arrow} $\overline{\alpha} : j \to i$ for each arrow $\alpha: i \to j$ in $Q$. We call the arrows in $\overline{Q}$ which are not reverse arrows \emph{ordinary arrows}. Given a \emph{weight} $\lambda \in \Bbbk^{Q_0}$ for $Q$, the corresponding \emph{deformed preprojective algebra} is the $\Bbbk$-algebra 
\begin{align*}
\Pi^\lambda(Q) \coloneqq \Bbbk \overline{Q}/I
\end{align*}
where $I$ is the two-sided ideal of $\Bbbk \overline{Q}$ with generators
\begin{align*}
\sum_{\small{\substack{\alpha \in Q_1 \\ t(\alpha)=i}}} \hspace{-3pt} \alpha \overline{\alpha} \hspace{3pt} - \sum_{\small{\substack{\alpha \in Q_1 \\ h(\alpha)=i }}} \hspace{-3pt} \overline{\alpha} \alpha \hspace{3pt} - \hspace{3pt} \lambda_i e_i
\end{align*}
for each vertex $i \in Q_0$. It is easy to see that we can equivalently define $I$ as being the two-sided ideal with the single generator
\begin{align*}
\sum_{\alpha \in Q_1} \hspace{-3pt} (\alpha \overline{\alpha} - \overline{\alpha} \alpha) \hspace{3pt} - \sum_{i \in Q_0} \lambda_i e_i.
\end{align*}
\end{defn}

It is helpful to think of a weight as a label from $\Bbbk$ at each vertex of $Q$, and we will often refer to $\lambda_i$ as the weight at vertex $i$.\\
\indent Now suppose that $\widetilde{Q}$ is extended Dynkin, with vertices and arrows (of its double) labelled as in Figure \ref{dynkinquivs}. Throughout this paper, it will be our convention that $\widetilde{Q}$ denotes an extended Dynkin quiver, while $Q$ will denote the corresponding Dynkin quiver obtained by removing the extending vertex $0$, where the orientation of the arrows comes from Figure \ref{dynkinquivs}. \\
\indent We are now able to define the algebras of interest to us. We write $\mathcal{O}^\lambda(\widetilde{Q})$ for the algebra
\begin{align*}
\mathcal{O}^\lambda(\widetilde{Q}) \coloneqq e_0 \Pi^\lambda(\widetilde{Q}) e_0.
\end{align*}
\noindent The elements of $\mathcal{O}^\lambda(\widetilde{Q})$ may be thought of as linear combinations of (equivalence classes of) paths in the double of $\widetilde{Q}$ which start and end at the extending vertex $0$.

\begin{figure}[p]
\centering

\begin{tikzpicture}[->,>=stealth,thick,scale=1.14]

\node at (-7-0.3,0) {$\widetilde{\mathbb{A}}_n$};

\node (0) at (4*360/6: 1.6cm) {0};
\node (1) at (3*360/6: 1.6cm) [thick] {1};
\node (2) at (2*360/6: 1.6cm) [thick] {2};
\node (3) at (1*360/6: 1.6cm) [thick] {\color{white}{4}};
\node (4) at (0*360/6: 1.6cm) [thick] {\color{white}{4}};
\node (5) at (-1*360/6: 1.6cm) [thick] {$n$};

\draw (0) to node[right]{$\alpha_0$} (1) ;
\draw (1) [out=4*360/6, in=3*360/6] to node[left]{$\overline{\alpha}_0$} (0);
\draw (1) to node[right,pos=0.4]{$\alpha_1$} (2);
\draw (2) [out=3*360/6, in=2*360/6] to node[left,shift={(0pt,4pt)}]{$\overline{\alpha}_1$} (1);
\draw (2) to node[below]{$\alpha_2$} (3);
\draw (3) [out=2*360/6, in=1*360/6] to node[above]{$\overline{\alpha}_2$} (2);
\draw (4) to node[left,pos=0.4]{$\alpha_{n-1}$} (5);
\draw (5) [out=0*360/6, in=-1*360/6] to node[right,shift={(0pt,-4pt)}]{$\overline{\alpha}_{n-1}$} (4);
\draw (5) to node[above]{$\alpha_n$} (0);
\draw (0) [out=-1*360/6, in=-2*360/6] to node[below]{$\overline{\alpha}_n$} (5);
\draw[dotted,-,shorten >=6pt,shorten <=6pt] (3) to (4);

\node at (5.5-0.3,0) {\phantom{a}};
\end{tikzpicture}

\begin{tikzpicture}[->,>=stealth,thick,scale=1.14]

\node at (-3.765-0.4,0.1) {$\widetilde{\mathbb{D}}_n$};

\node (0) at (-1.061,-1.061) {$0$};
\node (0a) at (-1.061-0.1,-1.061+0.05) {$\phantom{0}$};
\node (0b) at (-1.061+0.05,-1.061-0.15) {$\phantom{0}$};

\node (1) at (-1.061,1.061) {$1$};
\node (1a) at (-1.061+0.05,1.061+0.15) {$\phantom{0}$};
\node (1b) at (-1.061-0.13,1.061-0.02) {$\phantom{0}$};

\node (2) at (0,0) {$2$};
\node (2a) at (0,0.1) {$\phantom{2}$};
\node (2b) at (0,-0.1) {$\phantom{2}$};
\node (2c) at (-0.05,0.15) {$\phantom{2}$};
\node (2d) at (0.1,-0.05) {$\phantom{2}$};

\node (2e) at (0.1,0.05) {$\phantom{2}$};
\node (2f) at (-0.08,-0.12) {$\phantom{2}$};

\node (3) at (1.5,0) {$3$};
\node (3a) at (1.5,0.1) {$\phantom{3}$};
\node (3b) at (1.5,-0.1) {$\phantom{3}$};
\node (4) at (3,0) {$\phantom{3}$};
\node (4a) at (3,0.1) {$\phantom{3}$};
\node (4b) at (3,-0.1) {$\phantom{3}$};
\node (5) at (3.5,0) {$\cdots$};
\node (6) at (4,0) {$\phantom{3}$};
\node (6a) at (4,0.1) {$\phantom{3}$};
\node (6b) at (4,-0.1) {$\phantom{3}$};
\node (7) at (5.75,0) {$n{-}2$};
\node (7a) at (5.5,0.1) {$\phantom{3}$};
\node (7a2) at (5.5,-0.1) {$\phantom{3}$};


\node (7c) at (6-0.1,0.05) {$\phantom{2}$};
\node (7d) at (6+0.06,-0.15) {$\phantom{2}$};
\node (7e) at (6+0.05,0+0.15) {$\phantom{0}$};
\node (7f) at (6-0.13,0-0.02) {$\phantom{0}$};
\node (8) at (6.27+1.061,1.061) {$n{-}1$};

\node (8a) at (6+1.061-0.05,1.061+0.15) {$\phantom{0}$};
\node (8b) at (6+1.061+0.13,1.061-0.02) {$\phantom{0}$};
\node (9) at (6+1.061,-1.061) {$n$};
\node (9a) at (6+1.061+0.1,-1.061+0.05) {$\phantom{n}$};
\node (9b) at (6+1.061-0.08,-1.061-0.12) {$\phantom{n}$};

\node at (-3.765-0.4+12.5,0) {\phantom{a}};

\draw (0a) to (2c);
\draw (2d) to (0b);
\draw (1a) to (2e);
\draw (2f) to (1b);
\draw (2a) to node[above]{$\alpha_{2}$} (3a);
\draw (3b) to node[below]{$\overline{\alpha}_{2}$} (2b);
\draw (3a) to node[above]{$\alpha_{3}$} (4a);
\draw (4b) to node[below]{$\overline{\alpha}_{3}$} (3b);
\draw (6a) to node[above]{$\alpha_{n{-}3}$} (7a);
\draw (7a2) to node[below]{$\overline{\alpha}_{n{-}3}$} (6b);
\draw (7c) to (8a);
\draw (8b) to (7d);
\draw (7e) to (9a);
\draw (9b) to (7f);

\node at (-0.88,-0.30) {$\alpha_{0}$};
\node at (-0.15,-0.78) {$\overline{\alpha}_{0}$};
\node at (-0.88,0.30) {$\overline{\alpha}_{1}$};
\node at (-0.17,0.78) {$\alpha_{1}$};

\node at (6,0.78) {$\alpha_{n{-}1}$};
\node at (7.1,0.30) {$\overline{\alpha}_{n{-}1}$};
\node at (6.93,-0.30) {$\alpha_{n}$};
\node at (6.2,-0.78) {$\overline{\alpha}_{n}$};

\end{tikzpicture}

\begin{tikzpicture}[->,>=stealth,thick,scale=1.14]

\node at (-4.25,1.6) {$\widetilde{\mathbb{E}}_6$};

\node (2) at (0,0) {2};
\node (2a) at (0,0.1) {\phantom{2}};
\node (2b) at (0,-0.1) {\phantom{2}};
\node (3) at (1.5,0) {3};
\node (3a) at (1.5,0.1) {\phantom{2}};
\node (3b) at (1.5,-0.1) {\phantom{2}};

\node (4) at (3,0) {4};
\node (4a) at (3,0.1) {\phantom{2}};
\node (4b) at (3,-0.1) {\phantom{2}};
\node (4c) at (3.1,0) {\phantom{2}};
\node (4d) at (2.9,0) {\phantom{2}};

\node (5) at (4.5,0) {5};
\node (5a) at (4.5,0.1) {\phantom{2}};
\node (5b) at (4.5,-0.1) {\phantom{2}};
\node (6) at (6,0) {6};
\node (6a) at (6,0.1) {\phantom{2}};
\node (6b) at (6,-0.1) {\phantom{2}};

\node (0) at (3,3) {0};
\node (0a) at (3.1,3) {\phantom{2}};
\node (0b) at (2.9,3) {\phantom{2}};
\node (1) at (3,1.5) {1};
\node (1a) at (3.1,1.5) {\phantom{2}};
\node (1b) at (2.9,1.5) {\phantom{2}};

\draw (0a) to node[right]{$\alpha_{0}$} (1a);
\draw (1a) to node[right]{$\overline{\alpha}_{1}$} (4c);
\draw (1b) to node[left]{$\overline{\alpha}_{0}$} (0b);
\draw (4d) to node[left]{$\alpha_{1}$}(1b);

\draw (2a) to node[above]{$\alpha_{2}$} (3a);
\draw (3a) to node[above]{$\overline{\alpha}_{3}$} (4a);
\draw (4a) to node[above]{$\alpha_{4}$} (5a);
\draw (5a) to node[above]{$\overline{\alpha}_{5}$} (6a);

\draw (3b) to node[below]{$\overline{\alpha}_{2}$} (2b);
\draw (4b) to node[below]{$\alpha_{3}$} (3b);
\draw (5b) to node[below]{$\overline{\alpha}_{4}$} (4b);
\draw (6b) to node[below]{$\alpha_{5}$} (5b);

\node at (-4.25+12.5,0) {\phantom{a}};

\end{tikzpicture}

\begin{tikzpicture}[->,>=stealth,thick,scale=1.14]

\node at (-2.75,0.6) {$\widetilde{\mathbb{E}}_7$};

\node (0) at (0,0) {0};
\node (0a) at (0,0.1) {\phantom{0}};
\node (0b) at (0,-0.1) {\phantom{0}};
\node (1) at (1.5,0) {1};
\node (1a) at (1.5,0.1) {\phantom{0}};
\node (1b) at (1.5,-0.1) {\phantom{0}};
\node (2) at (3,0) {2};
\node (2a) at (3,0.1) {\phantom{0}};
\node (2b) at (3,-0.1) {\phantom{0}};
\node (3) at (4.5,0) {3};
\node (3a) at (4.5,0.1) {\phantom{0}};
\node (3b) at (4.5,-0.1) {\phantom{0}};
\node (3c) at (4.6,0) {\phantom{0}};
\node (3d) at (4.4,0) {\phantom{0}};
\node (4) at (6,0) {4};
\node (4a) at (6,0.1) {\phantom{0}};
\node (4b) at (6,-0.1) {\phantom{0}};
\node (5) at (7.5,0) {5};
\node (5a) at (7.5,0.1) {\phantom{0}};
\node (5b) at (7.5,-0.1) {\phantom{0}};
\node (6) at (9,0) {6};
\node (6a) at (9,0.1) {\phantom{0}};
\node (6b) at (9,-0.1) {\phantom{0}};
\node (7) at (4.5,1.5) {7};
\node (7a) at (4.6,1.5) {\phantom{0}};
\node (7b) at (4.4,1.5) {\phantom{0}};

\draw (0a) to (1a);
\draw (1a) to (2a);
\draw (2a) to (3a);
\draw (3a) to (4a);
\draw (4a) to (5a);
\draw (5a) to (6a);
\draw (1b) to (0b);
\draw (2b) to (1b);
\draw (3b) to (2b);
\draw (4b) to (3b);
\draw (5b) to (4b);
\draw (6b) to (5b);
\draw (7a) to (3c);
\draw (3d) to (7b);

\node at (-2.75+12.5,1.7) {\phantom{a}};

\end{tikzpicture}

\begin{tikzpicture}[->,>=stealth,thick,scale=1.14]

\node at (-2,0.6) {$\widetilde{\mathbb{E}}_8$};
\node at (0,1.7) {\phantom{a}};

\node (0) at (0,0) {0};
\node (0a) at (0,0.1) {\phantom{0}};
\node (0b) at (0,-0.1) {\phantom{0}};
\node (1) at (1.5,0) {1};
\node (1a) at (1.5,0.1) {\phantom{0}};
\node (1b) at (1.5,-0.1) {\phantom{0}};
\node (2) at (3,0) {2};
\node (2a) at (3,0.1) {\phantom{0}};
\node (2b) at (3,-0.1) {\phantom{0}};
\node (3) at (4.5,0) {3};
\node (3a) at (4.5,0.1) {\phantom{0}};
\node (3b) at (4.5,-0.1) {\phantom{0}};
\node (4) at (6,0) {4};
\node (4a) at (6,0.1) {\phantom{0}};
\node (4b) at (6,-0.1) {\phantom{0}};
\node (5) at (7.5,0) {5};
\node (5a) at (7.5,0.1) {\phantom{0}};
\node (5b) at (7.5,-0.1) {\phantom{0}};
\node (5c) at (7.6,0) {\phantom{0}};
\node (5d) at (7.4,0) {\phantom{0}};
\node (6) at (9,0) {6};
\node (6a) at (9,0.1) {\phantom{0}};
\node (6b) at (9,-0.1) {\phantom{0}};
\node (7) at (10.5,0) {7};
\node (7a) at (10.5,0.1) {\phantom{0}};
\node (7b) at (10.5,-0.1) {\phantom{0}};
\node (8) at (7.5,1.5) {8};
\node (8a) at (7.6,1.5) {\phantom{0}};
\node (8b) at (7.4,1.5) {\phantom{0}};

\draw (0a) to (1a);
\draw (1a) to (2a);
\draw (2a) to (3a);
\draw (3a) to (4a);
\draw (4a) to (5a);
\draw (5a) to (6a);
\draw (6a) to (7a);
\draw (1b) to (0b);
\draw (2b) to (1b);
\draw (3b) to (2b);
\draw (4b) to (3b);
\draw (5b) to (4b);
\draw (6b) to (5b);
\draw (7b) to (6b);
\draw (8a) to (5c);
\draw (5d) to (8b);

\end{tikzpicture}

\caption{Doubles of extended Dynkin graphs with the labelling of vertices and arrows that will be used throughout this paper. We have labelled the arrows only for those quivers in which we will need to refer to specific paths.} \label{dynkinquivs}
\end{figure}

\indent If $\lambda = \mathbf{0}$, then $\Pi(Q) \coloneqq \Pi^\lambda(Q)$ is the (undeformed) preprojective algebra of Gelfand and Ponomarev \cite{gp}, and in this case we also write $\mathcal{O}(\widetilde{Q}) \coloneqq \mathcal{O}^\lambda(\widetilde{Q})$. We will often write $\Pi^\lambda$ and $\mathcal{O}^\lambda$ (or $\Pi$ and $\mathcal{O}$ if $\lambda = \mathbf{0}$) when the corresponding quiver is either unimportant or understood. \\
\indent For our purposes, it is important to know precisely when the rings $\mathcal{O}^\lambda$ are noncommutative. This depends on the weight $\lambda$ and also on a vector $\delta \in \mathbb{N}^{\widetilde{Q}_0}$, which we now define. Let $G$ be the finite subgroup of $\text{SL}(2,\Bbbk)$ corresponding to $\widetilde{Q}$ by the McKay correspondence. Then each vertex of $\widetilde{Q}$ corresponds to an irreducible representation $W_i$ of $G$, and we set $\delta_i \coloneqq \dim_\Bbbk W_i$. If we number the vertices of $\widetilde{Q}$ as in Figure \ref{dynkinquivs}, then
\begin{align*}
&\widetilde{\mathbb{A}}_n: \quad \delta = ( \gap \underbrace{1,1, \dots, 1,1}_{n+1 \text{ times}} \gap ) \\
&\widetilde{\mathbb{D}}_n: \quad \delta = (1,1,\underbrace{2,2 \dots,2,2}_{n-3 \text{ times}},1,1) \\
&\widetilde{\mathbb{E}}_6: \quad \delta = (1,2,1,2,3,2,1) \\
&\widetilde{\mathbb{E}}_7: \quad \delta = (1,2,3,4,3,2,1,2) \\
&\widetilde{\mathbb{E}}_8: \quad \delta = (1,2,3,4,5,6,4,2,3).
\end{align*}

\indent Crawley-Boevey--Holland proved the following result:

\begin{thm}[{\cite[Theorem 0.1, Theorem 0.4 (1)]{cbh}}]
$\mathcal{O}^\lambda$ is commutative if and only if $\lambda \cdot \delta = \sum_{i \in \widetilde{Q}_0} \lambda_i \delta_i = 0$. In the case when $\lambda = \mathbf{0}$, $\mathcal{O}$ is isomorphic to the coordinate ring of the Kleinian singularity corresponding to $\widetilde{Q}$.
\end{thm}

In fact, by \cite[Lemma 2.2]{cbh}, one may always assume that $\lambda \cdot \delta$ is either $0$ or $1$. When $\lambda$ is a weight with $\lambda \cdot \delta = 0$ (so that $\mathcal{O}^\lambda$ is commutative), if we define $\lambda' = (\lambda_0 +1, \lambda_1, \dots , \lambda_n)$ then we consider $\mathcal{O}^{\lambda'}$ to be a noncommutative analogue of $\mathcal{O}^\lambda$. It is natural to ask if there is any relationship between the singularity categories of $\mathcal{O}^\lambda$ and $\mathcal{O}^{\lambda'}$. In fact, it follows from Theorem \ref{introthm0} that the singularity categories of these $\Bbbk$-algebras are triangle equivalent. \\
\indent When $\widetilde{Q} = \widetilde{\mathbb{A}}_n$, it can be checked that $\mathcal{O}^\lambda$ has a presentation of the form
\begin{align}
\frac{\Bbbk \langle x,y,z \rangle}{\left \langle
\begin{array}{cc}
{xz = (z+\lambda \cdot \delta)x}, & {xy = \prod_{i=0}^n \left(z + \sum_{j=1}^i \lambda_j \right)} \\
{yz = (z-\lambda \cdot \delta)y}, & {yx = \prod_{i=0}^n \left(z - \lambda \cdot \delta + \sum_{j=1}^i \lambda_j \right)}
\end{array}
\right \rangle}, \label{PresForO}
\end{align}
\noindent When $\mathcal{O}^\lambda$ is noncommutative, since there is no loss in generality in assuming $\lambda \cdot \delta = 1$, these are precisely the algebras considered by Hodges in \cite{hodges} and Bavula in \cite{bavula}, where in the latter they were called generalised Weyl algebras. \\
\indent In \cite[\S 7]{cbh}, the authors prove a number of results in the case where the weight $\lambda$ is \emph{dominant}, a term which we now define. Fix a total ordering $\prec$ on $\Bbbk$ which also satisfies the following:
\begin{enumerate}[{\normalfont (1)},topsep=0pt,itemsep=0pt]
\item If $a \prec b$, then $a+c \prec b+c$ for all $c \in \Bbbk$;
\item On the integers, $\prec$ coincides with the usual order; and
\item For any $a \in \Bbbk$, there exists $m \in \ZZ$ with $a \prec m$.
\end{enumerate}
For example, when $\Bbbk = \CC$ we may define $\prec$ by $z \prec z'$ if and only if $\real z < \real z'$, or $\real z = \real z'$ and $\imaginary z < \imaginary z'$. We then say that a weight $\lambda \in \Bbbk^{Q_0}$ is \emph{dominant} if $\lambda_i \succeq 0$ for all $i \in Q_0$. \\
\indent When $\lambda$ is dominant and $Q$ is (extended) Dynkin, Crawley-Boevey--Holland showed that it is easy to determine certain representation-theoretic properties of $\Pi^\lambda(Q)$. For example, we have the following useful result which we will use frequently in later sections:

\begin{lem}[{\cite[Lemma 7.1 (1)]{cbh}}] \label{PiofDynkin}
Suppose that $Q$ is Dynkin, and let $\lambda$ be a dominant weight for $Q$. Write $Q_\lambda$ for the full subquiver supported on those vertices $i$ with $\lambda_i = 0$. Then $\Pi^\lambda(Q) \cong \Pi(Q_\lambda)$. In particular, the projective $\Pi^\lambda(Q)$-modules are the modules $e_i \Pi^\lambda(Q)$, where $i$ is a vertex with $\lambda_i = 0$.
\end{lem}

\subsection{Restriction to quasi-dominant weights}
A weaker version of dominance will play an important role in this paper, which we now define:

\begin{defn} \label{quasidomdef}
If $\widetilde{Q}$ is extended Dynkin, we say that a weight $\lambda$ is \emph{quasi-dominant} if $\lambda_i \succeq 0$ for all $i \neq 0$, where $\prec$ is a total ordering on $\Bbbk$ as above.
\end{defn}

\indent It turns out that we are able to restrict attention to quasi-dominant weights for the remainder of this paper. We now state this as an assumption, before explaining why this is the case.

\begin{ass} \label{asss}
If $\lambda$ is a weight for an extended Dynkin quiver $\widetilde{Q}$, then we always assume that the weight $\lambda$ is quasi-dominant unless explicitly stated otherwise.
\end{ass}

To explain why this restriction is possible, we first recall a definition. Let $Q$ be a quiver, and let $C = 2I - A$ be the generalised Cartan matrix of $Q$, where $A$ is the adjacency matrix of the underlying graph of $Q$. For each loop-free vertex $i \in Q_0$, define the \emph{dual reflection} $r_i : \Bbbk^{Q_0} \to \Bbbk^{Q_0}$ by
\begin{align*}
(r_i \lambda)_j = \lambda_j - C_{ij} \lambda_i.
\end{align*}
It is easy to see that if $\widetilde{Q}$ is extended Dynkin then $\lambda \cdot \delta = (r_i \lambda) \cdot \delta$. We then have the following result, which appears in unpublished work of Boddington and Levy, \cite{levy}. 

\begin{lem} \label{bodlemma}
Suppose that $\lambda$ is a weight for an extended Dynkin quiver $\widetilde{Q}$, and let $\rho$ be a sequence of dual reflections at vertices other than the extending vertex $0$. Then $\mathcal{O}^\lambda \cong \mathcal{O}^{\rho(\lambda)}$.
\end{lem}

This is a strengthening of \cite[Lemma 7.9]{cbh}, in which the authors established only a Morita equivalence between these two rings, rather than an isomorphism. Combining Lemma \ref{bodlemma} with \cite[Lemma 7.8]{cbh}, we have the following result, which justifies the restriction given in Assumption \ref{asss}.

\begin{lem} \label{quasidomassum}
Suppose that $\lambda$ is a weight for an extended Dynkin quiver $\widetilde{Q}$. Then there exists a quasi-dominant weight $\lambda '$ with $\mathcal{O}^\lambda \cong \mathcal{O}^{\lambda '}$.
\end{lem}

\indent We will see later that this assumption allows one to easily read off a number of useful facts about the module category of $\mathcal{O}^\lambda$, and ultimately its singularity category as well. As a first example, if we restrict our attention to quasi-dominant weights then it is easy to detect whether $\mathcal{O}^\lambda$ is singular:

\begin{lem} \label{quasidomlem}
If $\lambda$ is a quasi-dominant weight for an extended Dynkin quiver $\widetilde{Q}$, then $\mathcal{O}^\lambda$ is singular if and only if $\lambda_i = 0$ for some $i \neq 0$.
\end{lem}
\begin{proof}
By \cite[Theorem 0.4 (4)]{cbh}, $\mathcal{O}^\lambda$ is singular if and only if $\lambda \cdot \alpha = 0$ for some \emph{Dynkin root} $\alpha$. The possible values of these Dynkin roots are not important to us; it suffices to know that they have the form $(0,\alpha') \in \mathbb{Z}^{n+1}$ where, in particular, $\alpha'$ has entirely non-negative or non-positive entries, and has at least one nonzero entry. In addition, $\varepsilon_i \in \mathbb{Z}^{n+1}$ for $1 \leqslant i \leqslant n$  is always a Dynkin root, where $\varepsilon_i$ is the $i$th coordinate vector (here the entries are indexed from $0$ to $n$). Therefore, if $\lambda_i = 0$ for some $i \neq 0$ then $\lambda \cdot \alpha = 0$ for the Dynkin root $\alpha = \varepsilon_i$, while if $\lambda_i \neq 0$ for all $i \neq 0$, then necessarily $\lambda \cdot \alpha \neq 0$ for all Dynkin roots $\alpha$. The result then follows.
\end{proof}


\section{The singularity category of $\mathcal{O}^\lambda(\widetilde{Q})$ as a $\Bbbk$-linear category} \label{singcatsec}

Our first step in determining $\stabMCM \mathcal{O}^\lambda$ is to determine its structure as an additive category, or indeed as a $\Bbbk$-linear category. We first identify an important module.

\begin{lem} \label{Pireflex}
$\Pi^\lambda e_0$ is a finitely generated $\mathcal{O}^\lambda$-module, and it satisfies $\End_{\mathcal{O}^\lambda}(\Pi^\lambda e_0) = \Pi^\lambda$. Moreover, $\Pi^\lambda e_0$ is maximal Cohen-Macaulay.
\end{lem}
\begin{proof}
The first part of the statement follows from \cite[Lemma 1]{montgomery}. To determine the endomorphism ring, first note that, by \cite[Lemma 1.4, Corollary 3.5]{cbh}, $\Pi^\lambda$ is Morita equivalent to a ring which is a maximal order and hence is itself a maximal order. The claim then follows from the results in \cite[Section 5.4]{dmv}. \\
\indent For the final claim, note that $\Pi^\lambda e_0$ is a reflexive $\mathcal{O}^\lambda$-module by \cite[Section 5.4]{dmv}. Therefore, since $\Pi^\lambda e_0$ is finitely generated, and since $\idim \mathcal{O}^\lambda \leqslant 2$ by \cite[Theorem 1.6]{cbh}, Lemma \ref{MCMreflexive} implies that $\Pi^\lambda e_0$ is maximal Cohen-Macaulay.
\end{proof}

Write $V_i = e_i \Pi^\lambda e_0$; we shall refer to these $\mathcal{O}^\lambda$-modules as \emph{vertex modules}, and they will play an important role in determining $\stabMCM \mathcal{O}^\lambda$. Using Lemma \ref{Pireflex}, we are able to calculate the Hom spaces between the vertex modules.

\begin{cor} \label{homs}
We have $\Hom_{\mathcal{O}^\lambda} (V_i,V_j) = e_j \Pi^\lambda e_i$.
\end{cor}
\begin{proof}
By Lemma \ref{Pireflex}, $\Pi^\lambda = \End_{\mathcal{O}^\lambda}(\Pi^\lambda e_0) = \bigoplus_{k,\ell} \Hom_{\mathcal{O}^\lambda}(e_k\Pi^\lambda e_0, e_\ell \Pi^\lambda e_0)$. Multiplying on the left by $e_j$ kills each Hom space with $\ell \neq j$, while multiplying on the right by $e_i$ kills each Hom space with $k \neq i$. It follows that 
\begin{gather*}
e_j \Pi^\lambda e_i = e_j \left( \bigoplus_{k,\ell} \Hom_{\mathcal{O}^\lambda}(V_k,V_\ell) \right) e_i =  e_j \Hom_{\mathcal{O}^\lambda}(V_i,V_j) \gap e_i = \Hom_{\mathcal{O}^\lambda}(V_i,V_j),
\end{gather*}
as claimed.
\end{proof}

This allows us to determine the stable endomorphism ring of $\Pi^\lambda e_0$. We fix some notation which will be used throughout the rest of this paper: write $Q_\lambda$ for the full subquiver of $\widetilde{Q}$ with vertex set $I_\lambda \coloneqq \{ i \in \{1, \dots , n\} \mid \lambda_i = 0\}$.

\begin{lem} \label{stabendpi1}
We have $\stabEnd_{\mathcal{O}^\lambda}(\Pi^\lambda e_0) \cong \Pi(Q_\lambda)$.
\end{lem}
\begin{proof}
Write $\mu = (\lambda_1, \dots, \lambda_n)$.  By Corollary \ref{homs}, we have that 
\begin{align*}
(\Pi^\lambda e_0)^* = \bigoplus_i \Hom_{\mathcal{O}^\lambda} (e_i \Pi^\lambda e_0, e_0 \Pi^\lambda e_0) = \bigoplus_i e_0 \Pi^\lambda e_i = e_0 \Pi^\lambda.
\end{align*}
Then, noting that $\Pi^\lambda e_0 \left( \Pi^\lambda e_0 \right)^* = \Pi^\lambda e_0 \Pi^\lambda$, we have
\begin{align*}
\stabEnd_{\mathcal{O}^\lambda}(\Pi^\lambda e_0) \cong \frac{\Pi^\lambda}{\Pi^\lambda e_0 \Pi^\lambda} \cong \Pi^{\mu}(Q).
\end{align*}
Since the entries of $\mu$ are all $\succeq 0$ by Assumption \ref{asss} and $Q$ is Dynkin, Lemma \ref{PiofDynkin} tells us that $\Pi^\mu(Q)$ is isomorphic to the preprojective algebra supported on the vertices $i$ of $Q$ with $\mu_i = 0$; that is, $\Pi^\mu(Q) \cong \Pi(Q_\lambda)$. Therefore $\stabEnd_{\mathcal{O}^\lambda}(\Pi^\lambda e_0) \cong \Pi(Q_\lambda)$.
\end{proof}

We are also able to determine when a vertex module is projective. It turns out that this is the case precisely when the corresponding vertex is deleted when passing from $\widetilde{Q}$ to $Q_\lambda$. 

\begin{lem} \label{isproj}
If $i=0$ or $\lambda_i \neq 0$, then $V_i$ is a projective $\mathcal{O}^\lambda$-module.
\end{lem}
\begin{proof}
When $i=0$ this is clear. So suppose that $i \neq 0$ and $\lambda_i \neq 0$. Then, as in the proof of Lemma \ref{stabendpi1}, $e_i = 0$ in $\Pi^\lambda / \Pi^\lambda e_0 \Pi^\lambda$ and so $e_i \in \Pi^\lambda e_0 \Pi^\lambda$. But then, using Corollary \ref{homs},
\begin{align*}
V_i V_i^* = e_i \Pi^\lambda e_0 \Pi^\lambda e_i \ni e_i^3 = e_i,
\end{align*}
where $e_i$ is the identity element of $\End_{\Pi^\lambda}(V_i) = e_i \Pi^\lambda e_i$, and so $V_i$ is projective by the Dual Basis Lemma (see \cite[(2.9)]{lam}). 
\end{proof}

It follows that the vertex modules $V_i$ satisfying $\lambda_i \neq 0$ are equal to the zero object in the singularity category, so that $\Pi^\lambda e_0$ and $\bigoplus_{i \in I_\lambda} V_i $ are isomorphic in $\stabMCM \mathcal{O}^\lambda$. When working in the stable module category, we will sometimes refer to those vertex modules whose corresponding weight is zero as non-projective vertex modules.

\begin{prop} \label{addPi} \leavevmode
\begin{enumerate}[{\normalfont (1)},leftmargin=*,topsep=0pt,itemsep=0pt]
\item$\MCM \mathcal{O}^\lambda = \add \Pi^\lambda e_0$.
\item $\stabMCM \mathcal{O}^\lambda =  \add \Pi^\lambda e_0 = \add \left( \bigoplus_{i \in I_\lambda} V_i \right)$.
\end{enumerate}
\end{prop}
\begin{proof} \leavevmode
\begin{enumerate}[{\normalfont (1)},wide=0pt,topsep=0pt,itemsep=0pt]
\item First note that $\mathcal{O}^\lambda$ is Gorenstein and that, using \cite[Theorem 1.5]{cbh},
\begin{align*}
\gldim \End_{\mathcal{O}^\lambda} (\Pi^\lambda e_0) = \gldim \Pi^\lambda \leqslant 2.
\end{align*}
Since $\Pi^\lambda e_0$ has $\mathcal{O}^\lambda$ as a direct summand, the first claim then follows from Proposition \ref{auslander}.
\item Part (1) immediately implies that $\stabMCM \mathcal{O}^\lambda = \add \Pi^\lambda e_0 = \add \left( \bigoplus_{i} V_i \right)$. But projective modules become the zero object when passing to the stable module category, so the result follows by Lemma \ref{isproj}. \qedhere
\end{enumerate}
\end{proof}

\indent We recall that an additive category is said to be \emph{Krull-Schmidt} if every object decomposes into a finite direct sum of objects, each of which has a local endomorphism ring. By \cite[Theorem 4.2]{krause}, this decomposition is unique up to reordering.

\begin{thm} \label{addequiv}
The functor $\stabHom_{\mathcal{O}^\lambda}(\Pi^\lambda e_0,-)$ induces a $\Bbbk$-linear equivalence
\begin{align*}
\stabMCM \mathcal{O}^\lambda \simeq \proj \Pi(Q_\lambda).
\end{align*}
\end{thm}
\begin{proof}
By \cite[Proposition 2.3]{krause}, the functor
\begin{align*}
\stabHom_{\mathcal{O}^\lambda}(\Pi^\lambda e_0, -): \stabmod \mathcal{O}^\lambda \to \mod \stabEnd_{\mathcal{O}^\lambda}(\Pi^\lambda e_0) = \mod \Pi(Q_\lambda)
\end{align*}
induces a fully faithful $\Bbbk$-linear functor $\add \Pi^\lambda e_0 \to \proj \Pi(Q_\lambda)$, where $\add \Pi^\lambda e_0 = \stabMCM \mathcal{O}^\lambda$ by Proposition \ref{addPi}. Since $\Pi(Q_\lambda)$ is finite-dimensional \cite[Proposition 2.1]{erdmann}, $\mod \Pi(Q_\lambda)$ is Krull-Schmidt and hence so too is $\proj \Pi(Q_\lambda)$. Therefore, to establish essential surjectivity of the functor $\stabHom_{\mathcal{O}^\lambda}(\Pi^\lambda e_0,-)$, it suffices to show that we can hit each indecomposable projective $e_i \Pi(Q_\lambda)$, where $i \in I_\lambda$. Indeed, we have
\begin{align*}
\stabHom_{\mathcal{O}^\lambda}(\Pi^\lambda e_0, V_i) = \frac{e_i \Pi^\lambda}{e_i \Pi^\lambda e_0 \Pi^\lambda} = \frac{e_i \Pi^\lambda}{e_i \Pi^\lambda \cap \Pi^\lambda e_0 \Pi^\lambda} = e_i \gap \frac{ \Pi^\lambda}{\Pi^\lambda e_0 \Pi^\lambda} = e_i \Pi(Q_\lambda),
\end{align*}
and so the functor is also essentially surjective. We therefore have the claimed equivalence.
\end{proof}

It follows that $\stabMCM \mathcal{O}^\lambda$ is nontrivial if and only if $\lambda_i = 0$ for some $i \neq 0$ which, by Lemma \ref{quasidomlem}, happens precisely when $\mathcal{O}^\lambda$ is singular; this is consistent with the more general fact that $\Dsg(R)$ is nontrivial if and only if $R$ is singular. Moreover, the vertex modules $V_i$ with $i=0$, or with $i \neq 0$ and $\lambda_i \neq 0$, are those which are projective and hence vanish in $\stabMCM \mathcal{O}^\lambda$. This is reflected by the fact that these are the vertices which are deleted to obtain $Q_\lambda$. \\
\indent  As an immediate consequence of (the proof of) Theorem \ref{addequiv}, we have the following result:

\begin{cor} \label{krull}
$\stabMCM \mathcal{O}^\lambda$ is a Krull-Schmidt category.
\end{cor}

\begin{rem} 
By Proposition \ref{addPi}, the objects of $\stabMCM \mathcal{O}^\lambda$ are direct summands of finite direct sums of the non-projective vertex modules. Since these vertex modules are indecomposable and $\stabMCM \mathcal{O}^\lambda$ is Krull-Schmidt, in fact every object of $\stabMCM \mathcal{O}^\lambda$ is isomorphic to a finite direct sum of vertex modules.
\end{rem}

\indent The following two corollaries are then immediate from Theorem \ref{addequiv}:

\begin{cor} \label{addequiv3}
Suppose that $\widetilde{Q}$ is an extended Dynkin quiver and $Q_\lambda = Q^{(1)} \sqcup \dots \sqcup Q^{(r)}$ is a disjoint union of connected quivers $Q^{(i)}$, which are therefore necessarily Dynkin. Then there is a $\Bbbk$-linear equivalence
\begin{gather*}
\stabMCM \mathcal{O}^\lambda \simeq \bigoplus_{i=1}^r \proj \Pi(Q^{(i)}).
\end{gather*}
\end{cor}

\begin{cor} \label{addequiv2}
Let $\widetilde{Q}$ and $\widetilde{Q}'$ be extended Dynkin quivers (not necessarily of the same type) and let $\lambda$ and $\lambda'$ be quasi-dominant weights for $\widetilde{Q}$ and $\widetilde{Q}'$, respectively. If $Q_\lambda \cong Q'_{\lambda'}$ then there is a $\Bbbk$-linear equivalence 
\begin{gather*}
\pushQED{\qed} 
\stabMCM \mathcal{O}^\lambda(\widetilde{Q}) \simeq \stabMCM \mathcal{O}^{\lambda'}(\widetilde{Q}'). \qedhere
\popQED
\end{gather*}
\end{cor}

It is illustrative to apply Theorem \ref{addequiv} (and its corollaries) to an example:

\begin{example}

Suppose that $\widetilde{Q} = \widetilde{\mathbb{A}}_5$, and consider the deformation $\mathcal{O}^\lambda$ where the weight $\lambda$ is indicated in red on the left hand quiver below:

\begin{figure}[h]
\centering
\begin{tikzpicture}[->,>=stealth,thick,scale=1]

\node (0) at (4*360/6: 1.6cm) {0};
\node (1) at (3*360/6: 1.6cm) [thick] {1};
\node (2) at (2*360/6: 1.6cm) [thick] {2};
\node (3) at (1*360/6: 1.6cm) [thick] {3};
\node (4) at (0*360/6: 1.6cm) [thick] {4};
\node (5) at (-1*360/6: 1.6cm) [thick] {5};
\node[red] at (4*360/6: 2.3cm) [thick] {$\lambda_0$};
\node[red] at (3*360/6: 2.3cm) [thick] {$0$};
\node[red] at (2*360/6: 2.3cm) [thick] {$0$};
\node[red] at (1*360/6: 2.3cm) [thick] {$0$};
\node[red] at (0*360/6: 2.7cm) [thick] {$\lambda_4 \succ 0$};
\node[red] at (5*360/6: 2.3cm) [thick] {$0$};

\draw (0) to (1) ;
\draw (1) [out=4*360/6, in=3*360/6] to (0);
\draw (1) to  (2);
\draw (2) [out=3*360/6, in=2*360/6] to  (1);
\draw (2) to (3);
\draw (3) [out=2*360/6, in=1*360/6] to (2);
\draw (3) to (4);
\draw (4) [out=1*360/6, in=0*360/6] to (3);
\draw (4) to (5);
\draw (5) [out=0*360/6, in=-1*360/6] to (4);
\draw (5) to (0);
\draw (0) [out=-1*360/6, in=-2*360/6] to (5);

\node at (4.4,0) {\scalebox{2}{$\leadsto$}};

\begin{scope}[xshift=247pt]

\node[mygray] (0) at (4*360/6: 1.6cm) {0};
\node (1) at (3*360/6: 1.6cm) [thick] {1};
\node (2) at (2*360/6: 1.6cm) [thick] {2};
\node (3) at (1*360/6: 1.6cm) [thick] {3};
\node[mygray] (4) at (0*360/6: 1.6cm) [thick] {4};
\node (5) at (-1*360/6: 1.6cm) [thick] {5};

\draw[mygray] (0) to (1) ;
\draw[mygray] (1) [out=4*360/6, in=3*360/6] to (0);
\draw (1) to  (2);
\draw (2) [out=3*360/6, in=2*360/6] to  (1);
\draw (2) to (3);
\draw (3) [out=2*360/6, in=1*360/6] to (2);
\draw[mygray] (3) to (4);
\draw[mygray] (4) [out=1*360/6, in=0*360/6] to (3);
\draw[mygray] (4) to (5);
\draw[mygray] (5) [out=0*360/6, in=-1*360/6] to (4);
\draw[mygray] (5) to (0);
\draw[mygray] (0) [out=-1*360/6, in=-2*360/6] to (5);

\end{scope}

\end{tikzpicture}
\end{figure}

\noindent By Theorem \ref{addequiv3}, there is a $\Bbbk$-linear equivalence $\stabMCM \mathcal{O}^\lambda \simeq \proj \Pi(\mathbb{A}_3) \oplus \proj \Pi(\mathbb{A}_1)$. In more suggestive notation, we can write this equivalence as
\begin{align*}
\Dsg (\mathcal{O}^\lambda) \simeq \Dsg(R_{\mathbb{A}_3}) \oplus \Dsg(R_{\mathbb{A}_1}),
\end{align*}
and so it is sensible to consider $\mathcal{O}^\lambda$ as having an $\mathbb{A}_3$ singularity and an $\mathbb{A}_1$ singularity. \\
\indent If we more concretely set $\lambda = (-1,0,0,0,1,0)$ (respectively, $\lambda = (0,0,0,0,1,0)$) then $\mathcal{O}^\lambda$ is commutative (respectively, noncommutative). In this case we can use (\ref{PresForO}) to write down a presentation for $\mathcal{O}^\lambda$. In particular, we have a $\Bbbk$-linear equivalence
\begin{align*}
\Dsg \frac{\Bbbk[x,y,z]}{\langle xy - z^4(z+1)^2 \rangle} \simeq
\Dsg
\frac{\Bbbk \langle x,y,z \rangle}{\left \langle
\begin{array}{cc}
{xz = (z+1)x}, & {xy = z^4(z+1)^2} \\
{yz = (z-1)y}, & {yx = (z-1)^4 z^2}
\end{array}
\right \rangle}.
\end{align*}
\end{example}


\section{The singularity category of $\mathcal{O}^\lambda(\widetilde{Q})$ as a triangulated category} \label{SingChapterTriSec}

We are now able to complete the proof of Theorem \ref{introthm0} which gives a complete description of the singularity category of $\mathcal{O}^\lambda(\widetilde{Q})$. To do this, we need to show that the induced triangulated structures on the right hand sides of the $\Bbbk$-linear equivalences
\begin{gather}
\Dsg( \mathcal{O}^\lambda(\widetilde{Q})) \simeq \stabMCM \mathcal{O}^\lambda(\widetilde{Q})) \simeq \bigoplus_{i=1}^r \proj \Pi(Q^{(i)}), \label{equivalences1} \\
\bigoplus_{i=1}^{r} \Dsg (R_{Q^{(i)}}) \simeq \bigoplus_{i=1}^r \proj \Pi(Q^{(i)}) \label{equivalences 2}
\end{gather}
are the same. We achieve this in by showing that the summands $\proj \Pi(Q^{(i)})$ in (\ref{equivalences1}) are in fact triangulated subcategories of $\bigoplus_{i=1}^r \proj \Pi(Q^{(i)})$, and that the triangulated structure on each $\proj \Pi(Q^{(i)})$ is essentially unique. To establish the latter of these, we use the following result:

\begin{thm}[{\cite[Corollary 2]{keller18}}] \label{amiotthm}
Let $\mathcal{T}$ and $\mathcal{T}'$ be Krull-Schmidt $\Bbbk$-linear triangulated categories which are finite, connected, algebraic and standard. If $\mathcal{T}$ and $\mathcal{T}'$ are equivalent as $\Bbbk$-linear categories, then they are in fact equivalent as triangulated categories.
\end{thm}

We note that, if $Q$ is Dynkin, $\proj \Pi(Q)$ (and the $\Bbbk$-linearly equivalent category $\Dsg(R_Q)$) are finite, connected, and standard since they are $\Bbbk$-linearly equivalent to certain orbit categories which are known to have these properties (see \cite[Remark 5.9]{ari}). Therefore, if we can show that each $\proj \Pi(Q^{(i)})$ is an algebraic triangulated subcategory under the $\Bbbk$-linear equivalence (\ref{equivalences1}), then each $\Bbbk$-linear equivalence $\proj \Pi(Q^{(i)}) \simeq \Dsg(R_{Q^{(i)}})$ is in fact a triangle equivalence, which will prove Theorem \ref{introthm0} from the introduction. \\
\indent We must first show that the translation functor $\Sigma$ induced on the right hand side of (\ref{equivalences1}) \emph{preserves connected components}, in the sense that it restricts to an autoequivalence of each of the subcategories $\proj \Pi(Q^{(i)})$. Writing $P_i$ for the indecomposable projective $\Pi(Q_\lambda)$-module corresponding to vertex $i$ in $Q_\lambda$, it follows from the Krull-Schmidt property of $\proj \Pi(Q_\lambda)$ that $\Sigma$ permutes the $P_i$. We write $\sigma$ for the induced permutation of the vertices. This allows us to make the following observation:

\begin{lem} \label{graphautolem}
With the above setup, $\sigma$ is a graph automorphism of $Q_\lambda$.
\end{lem}
\begin{proof}
First note that the spaces $\Hom_{\Pi(Q_\lambda)}(P_i, P_j) = e_j \Pi(Q_\lambda) e_i$ can be graded by path length, and that vertex $i$ and vertex $j$ are adjacent in $Q_\lambda$ if and only if there is a degree 1 morphism in $\Hom_{\Pi(Q_\lambda)}(P_i, P_j)$. Applying $\Sigma$, this is equivalent to $\Hom_{\Pi(Q_\lambda)}(P_{\sigma(i)}, P_{\sigma(j)})$ containing a degree 1 morphism, which happens if and only if $\sigma(i)$ and $\sigma(i)$ are adjacent in $Q_\lambda$. That is, $\sigma$ is a graph automorphism of $Q_\lambda$.
\end{proof}

If the $Q^{(i)}$ are pairwise non-isomorphic, then the fact that the induced translation functor has to be a graph automorphism forces it to preserve connected components, as required. This leaves only the cases where some of the $Q^{(i)}$ are isomorphic, and one might hope to abstractly prove that the translation functor preserves connected components. Unfortunately, the following example shows that one should not expect this to be the case.

\begin{example}
Let $T$ be a Krull-Schmidt $\Bbbk$-linear category with only two indecomposable objects $U$ and $V$, and suppose these objects satisfy $\Hom_T(U,V) = 0 = \Hom_T(V,U)$ and $\End_T(U) = \Bbbk = \End_T(V)$. For example, this is the case for $\stabMCM \mathcal{O}^\lambda(\widetilde{\mathbb{A}}_3)$ when $(\lambda_0,\lambda_1,\lambda_2,\lambda_3) = (0,0,1,0)$, since in this case it is $\Bbbk$-linearly equivalent to $\proj \Pi(\mathbb{A}_1) \oplus \proj \Pi(\mathbb{A}_1)$.  This category has two possible triangulated structures: the first has $\Sigma = \id$, and the distinguished triangles are isomorphic to direct sums and rotations of
\begin{align*}
U \xrightarrow{\id} U \to 0 \to U \quad \text{and} \quad V \xrightarrow{\id} V \to 0 \to V,
\end{align*}
and the second option has $\Sigma \gap U = V$ and $\Sigma \gap V = U$, and the distinguished triangles are isomorphic to direct sums and rotations of 
\begin{align*}
U \xrightarrow{\id} U \to 0 \to V.
\end{align*}
The first example decomposes into a direct sum of two triangulated subcategories, while the second example does not.
\end{example}

\indent While the above example shows that one should not expect to be able to abstractly prove that the translation functor preserves connected components, this is essentially the only counterexample. The following proof is due to Jeremy Rickard, and we thank him for allowing us to reproduce it:

\begin{lem} \label{onlycounterexample}
Suppose that $\mathcal{T}$ is a Krull-Schmidt $\Bbbk$-linear triangulated category with finitely many indecomposables which decomposes as a $\Bbbk$-linear category as
\begin{align*}
\mathcal{T} = \bigoplus_{i=1}^n \mathcal{T}_i.
\end{align*}
Suppose that the translation functor $\Sigma$ satisfies $\Sigma \gap \mathcal{T}_i = \mathcal{T}_j$ for some $i \neq j$. Then $\mathcal{T}_i$ and $\mathcal{T}_j$ each have only one isoclass of indecomposable objects.
\end{lem}
\begin{proof}
Let $\alpha: X \to Y$ be a nonzero morphism between two indecomposable objects of $\mathcal{T}_i$, and complete to a triangle
\begin{align*}
X \xrightarrow{\alpha} Y \xrightarrow{\beta} Z \xrightarrow{\gamma} \Sigma \gap X,
\end{align*}
where $\Sigma \gap X \in \mathcal{T}_j$ by assumption. We claim that every indecomposable summand of $Z$ lies in $\mathcal{T}_j$. To this end, suppose that $Z = Z' \oplus Z''$ where $Z' \in \mathcal{T}_j$ and $Z'' \in \bigoplus_{k \neq j} \mathcal{T}_k$, and write $\gamma = (\gamma',0)$. The map $\gamma'$ gives rise to a triangle $Z' \xrightarrow{\gamma'} \Sigma \gap X \to Y' \to \Sigma \gap Z'$ and rotating yields the triangle $X \to \Sigma^{-1} Y' \to Z' \xrightarrow{\gamma'} \Sigma \gap X$. The direct sum of this triangle with the triangle $0 \to Z'' \to Z'' \to 0$ is a triangle isomorphic to the triangle $X \to Y \to Z \to \Sigma \gap X$, and so $Y \cong \Sigma^{-1} Y' \oplus Z''$. By indecomposability of $Y$, we therefore have $\Sigma^{-1} Y' = 0$ or $Z'' = 0$. If $\Sigma^{-1} Y' = 0$ then $Y \cong Z''$ and $\Sigma \gap X \cong Z'$. Our original triangle becomes
\begin{align*}
X \xrightarrow{\alpha} Z'' \to \Sigma X \oplus Z'' \to \Sigma X
\end{align*}
which is isomorphic to the direct sum of the triangles $X \to 0 \to \Sigma X \to \Sigma X$ and $0 \to Z'' \to Z'' \to 0$. This means that $\alpha$ is the zero map, contrary to our assumption, and so we must have $Z'' = 0$, establishing the claim. Now, since every indecomposable summand of $Z$ lies in $\mathcal{T}_j$, $\beta$ is the zero map. Applying $\Hom(Y,-)$, we get an exact sequence
\begin{align*}
\Hom(Y,X) \xrightarrow{\alpha \circ -} \Hom(Y,Y) \to \Hom(Y,Z)
\end{align*}
where the last term is $0$. By exactness, there exists $\alpha' : Y \to X$ with $\alpha \alpha' = \id_Y$. Since $\mathcal{T}$ is Krull-Schmidt the endomorphism ring of $X$ is local, which implies that the idempotent map $\alpha' \alpha$ is a unit and therefore equal to $\id_X$. Therefore $\alpha : X \to Y$ is an isomorphism, and so $\mathcal{T}_i$ (and hence $\mathcal{T}_j$) has only one indecomposable object, up to isomorphism. 
\end{proof}

\indent Therefore, to show that the induced translation functor on $\bigoplus_{i=1}^r \proj \Pi(Q^{(i)})$ from the second $\Bbbk$-linear equivalence in (\ref{equivalences1}) preserves connected components, we only need to consider the case when there exist $Q^{(i)}$ and $Q^{(j)}$, $i \neq j$, with $Q^{(i)} = \mathbb{A}_1 = Q^{(j)}$. It suffices to show that, for the corresponding objects $V_i, V_j \in \stabMCM \mathcal{O}^\lambda(\widetilde{Q})$, we have $\Sigma V_i = V_i$ and $\Sigma V_j = V_j$. To this end, we first have the following result:

\begin{prop} \label{obvexactsequences}
Let $Q$ be a non-Dynkin quiver with no oriented cycles, and with vertices labelled $\{0, 1, \dots n \}$. Write $\Pi(Q)$ for the preprojective algebra of $Q$, and write $V_i = e_i \Pi(Q) e_0$, which is a right $e_0 \Pi(Q) e_0$-module. Then, for any $i \neq 0$, there exists a short exact sequence of $e_0 \Pi(Q) e_0$-modules
\begin{align*}
0 \to V_i \to \bigoplus_{j \in \partial i} V_j \to V_i \to 0,
\end{align*}
where $\partial i$ is the set of vertices adjacent to $i$ in $Q$.
\end{prop}
\begin{proof}
By \cite[Proposition 4.2]{brenner}, there is an exact sequence of $\Pi(Q)$-modules
\begin{align}
0 \to e_i \Pi(Q) \to \bigoplus_{j \in \partial i} e_j \Pi(Q) \to e_i \Pi(Q) \to S_i \to 0, \label{bremnerexact}
\end{align}
where $S_i$ is the simple module at vertex $i$. Noting that $e_0 \Pi(Q)$ is a direct summand of $\Pi(Q)$ and hence projective, applying $\Hom_{\Pi(Q)}(e_0 \Pi(Q),-)$ yields an exact sequence
\begin{align*}
0 &\to \Hom_{\Pi(Q)}(e_0 \Pi(Q),e_i \Pi(Q)) \to \bigoplus_{j \in \partial i} \Hom_{\Pi(Q)}(e_0 \Pi(Q),e_j \Pi(Q)) \\ 
&\to \Hom_{\Pi(Q)}(e_0 \Pi(Q),e_i \Pi(Q)) \to \Hom_{\Pi(Q)}(e_0 \Pi(Q),S_i) \to 0.
\end{align*}
We also have $\Hom_{\Pi(Q)}(e_0 \Pi(Q),e_k \Pi(Q)) = V_k$ and, since $i \neq 0$, $\Hom_{\Pi(Q)}(e_0 \Pi(Q),S_i) = 0$. We therefore have exactness of
\begin{align*}
0 \to V_i \to \bigoplus_{j \in \partial i} V_j \to V_i \to 0,
\end{align*}
as claimed.
\end{proof}

In particular this result holds for extended Dynkin quivers, where we remark that if we wish to apply it to an $\widetilde{\mathbb{A}}_n$ quiver then we must orient the arrows so that there are no oriented cycles; this does not change the isomorphism class of $\Pi(\widetilde{\mathbb{A}}_n)$ \cite[Lemma 2.2]{cbh}.

\begin{rem}
The above result may or may not fail for Dynkin quivers, depending on how the vertices are labelled. For example, when $Q = \mathbb{A}_3$ where the vertices are labelled as follows,

\begin{figure}[h]
\centering
\begin{tikzpicture}[-,thick,scale=1]

\node (0) at (-3,0) {$0$};
\node (1) at (-1.5,0) {$1$};
\node (2) at (0,0) {$2$};

\draw (0) to (1) to (2);

\end{tikzpicture}
\end{figure}

\noindent then the complexes of interest to us are 
\begin{align*}
0 \to V_1 \to V_0 \oplus V_2 \to V_1 \to 0, \quad \text{and} \quad 0 \to V_2 \to V_1 \to V_2 \to 0.
\end{align*}
\noindent Since $\dim_\Bbbk V_0 = 1$, $\dim_\Bbbk V_1 = 1$, and $\dim_\Bbbk V_2 =1$, the first of these is exact while the second is not. If instead we label the vertices of $Q$ as follows,

\begin{figure}[h]
\centering
\begin{tikzpicture}[-,thick,scale=1]

\node (1) at (-3,0) {$1$};
\node (0) at (-1.5,0) {$0$};
\node (2) at (0,0) {$2$};

\draw (1) to (0) to (2);

\end{tikzpicture}
\end{figure}

\noindent then the complexes of interest to us are
\begin{align*}
0 \to V_1 \to V_0  \to V_1 \to 0, \quad \text{and} \quad 0 \to V_2 \to V_0 \to V_2 \to 0,
\end{align*}
and both of these are exact since $\dim_\Bbbk V_0 = 2$, $\dim_\Bbbk V_1 = 1$, and $\dim_\Bbbk V_2 =1$.
\end{rem}

We now use Proposition \ref{obvexactsequences} to show that the induced translation functor on $\bigoplus_{i=1}^r \proj \Pi(Q^{(i)})$ preserves connected components.

\begin{prop} \label{connectedcomponents}
Let $\widetilde{Q}$ be an extended Dynkin quiver and $\lambda$ be a quasi-dominant weight for $\widetilde{Q}$. Write $Q_\lambda = Q^{(1)} \sqcup \dots \sqcup Q^{(r)}$ as a disjoint union of connected quivers $Q^{(i)}$, which are therefore necessarily Dynkin. Consider the triangulated structure on $\bigoplus_{i=1}^r \proj \Pi(Q^{(i)})$ induced by the $\Bbbk$-linear equivalence
\begin{gather*}
\stabMCM \mathcal{O}^\lambda \simeq \bigoplus_{i=1}^r \proj \Pi(Q^{(i)})
\end{gather*}
of Corollary \ref{addequiv3}, and let $\Sigma$ be the translation functor. Then each $\proj \Pi(Q^{(i)})$ is invariant under $\Sigma$.
\end{prop}
\begin{proof}
By Lemma \ref{onlycounterexample} and the discussion following it, the only situation in which there exist $\proj \Pi(Q^{(i)})$ which are not necessarily invariant under $\Sigma$ is when we have multiple $Q^{(i)}$ equal to $\mathbb{A}_1$. Working in $\stabMCM \mathcal{O}^\lambda(\widetilde{Q})$, this happens if and only if there is some vertex $i$ with $\lambda_i = 0$, and if $j$ is adjacent to $i$ then either $j = 0$ or $\lambda_j \neq 0$; in particular, the modules $V_j$ corresponding to these vertices are projective as $\mathcal{O}^\lambda(\widetilde{Q})$-modules by Lemma \ref{isproj}. By Proposition \ref{obvexactsequences}, we have an exact sequence of $e_0 \Pi(\widetilde{Q}) e_0$-modules
\begin{align}
0 \to V_i \xrightarrow{\phi} \bigoplus_{j \in \partial i} V_j \xrightarrow{\psi} V_i \to 0. \label{obvses}
\end{align}
Now consider (\ref{obvses}) as a sequence of modules over $\mathcal{O}^\lambda(\widetilde{Q})$. It is a complex since the composition $\psi \phi$ is equal to the (undeformed) preprojective relation at vertex $i$, which is equal to $\lambda_i e_i = 0$. Filtering $\Pi(\widetilde{Q})$ and $\mathcal{O}^\lambda(\widetilde{Q})$ by path length we obtain a sequence of associated graded modules, which is in fact the exact sequence (\ref{obvses}). It is standard (see \cite[Proposition 7.6.14]{mandr}) that this implies that (\ref{obvses}) is exact as a sequence of modules over $\mathcal{O}^\lambda(\widetilde{Q})$. To summarise, we have an exact sequence of $\mathcal{O}^\lambda(\widetilde{Q})$-modules
\begin{align*}
0 \to V_i \to \bigoplus_{j \in \partial i} V_j \to V_i \to 0
\end{align*}
whose middle term is projective. It follows from the definition of the translation functor that $\Sigma \gap V_i = V_i$ in $\stabMCM \mathcal{O}^\lambda(\widetilde{Q})$. Thus each $\proj \Pi(Q^{(i)})$ is invariant under the induced translation functor on $\bigoplus_{i=1}^r \proj \Pi(Q^{(i)})$.
\end{proof}

We now seek to prove Theorem \ref{introthm0}. Retaining all of the above notation, for each $1 \leqslant i \leqslant r$, define 
\begin{align*}
\mathcal{W}_i \coloneqq \{ V_j \mid j \in Q_0^{(i)} \}, \quad \mathcal{C}_i \coloneqq \add \bigg(V_0 \oplus \bigoplus_{j \in Q_0^{(i)}} V_j \bigg) \quad \text{and} \quad \mathcal{T}_i \coloneqq \add \bigg( \bigoplus_{j \in Q_0^{(i)}} V_j \bigg),
\end{align*}
where the latter two are viewed as subcategories of $\MCM \mathcal{O}^\lambda$ and $\stabMCM \mathcal{O}^\lambda$, respectively.  It will also be convenient to write 
\begin{align*}
M_i = \bigoplus_{j \in (Q_\lambda)_0 \setminus Q_0^{(i)}} V_j,
\end{align*}
and to set
\begin{align*}
\mathcal{W}_i^{\gap \text{c}} \coloneqq \{ V_j \mid j \in (Q_\lambda)_0 \setminus Q_0^{(i)} \}, \quad \mathcal{C}_i^{\hspace{1pt}\text{c}} \coloneqq \add (V_0 \oplus M_i ) \quad \text{and} \quad \mathcal{T}_i^{\hspace{1pt}\text{c}} \coloneqq \add M_i.
\end{align*}
\indent Observe that we can decompose $\stabMCM \mathcal{O}^\lambda$ as
\begin{align*}
\stabMCM \mathcal{O}^\lambda = \bigoplus_{i=1}^r \mathcal{T}_i
\end{align*}
as $\Bbbk$-linear categories. We wish to show that this is also a decomposition into triangulated subcategories. To do this, we first prove a result which shows that the $\mathcal{C}_i$ are Frobenius subcategories of the Frobenius category $\MCM \mathcal{O}^\lambda$. We call a subcategory $\mathcal{B}$ of an exact category $\mathcal{A}$ \emph{extension-closed} if whenever we have a conflation $X \to Y \to Z$ with $X, Z \in \mathcal{B}$ then necessarily $Y \in \mathcal{B}$. Furthermore, an extension-closed subcategory $\mathcal{B}$ is called \emph{admissible} provided that every $B \in \mathcal{B}$ fits into conflations $B \to P \to B'$ and $B'' \to Q \to B$ with $B',B'' \in \mathcal{B}$ and where $P, Q$ are projective in $\mathcal{A}$. We remark that an admissible subcategory of a Frobenius category is itself Frobenius; see \cite[\S 2]{frob}.

\begin{lem} \label{frobsubcat}
For each $i$, the subcategory $\mathcal{C}_i$ satisfies the following property: if $X \to Y \to Z$ is a conflation in $\MCM \mathcal{O}^\lambda$ such that two of the three objects are in $\mathcal{C}_i$, then the third object is also in $\mathcal{C}_i$. Consequently, $\mathcal{C}_i$ is a Frobenius subcategory of $\MCM \mathcal{O}^\lambda$.
\end{lem}
\begin{proof}
We only show that if $X \to Y \to Z$ is a conflation with $X,Y \in \mathcal{C}_i$ then $Z \in \mathcal{C}_i$, with the other cases being similar. So suppose that we have such a conflation. Since $\stabMCM \mathcal{O}^\lambda \simeq \proj \Pi(Q_\lambda)$ and this category is Krull-Schmidt, we have $Z \oplus P \cong U \oplus U' \oplus Q$ in $\MCM \mathcal{O}^\lambda$, where $U \in \mathcal{W}_i$, $U' \in \mathcal{W}_i^{\gap \text{c}}$, and $P,Q$ are projective. This conflation gives rise to a triangle $X \to Y \to Z \to \Sigma X$ in $\stabMCM \mathcal{O}^\lambda$, and applying the functor $\stabHom_{\mathcal{O}^\lambda}(M_i,-)$ yields an exact sequence
\begin{align*}
\stabHom_{\mathcal{O}^\lambda}(M_i,Y) \to \stabHom_{\mathcal{O}^\lambda}(M_i,Z) \to \stabHom_{\mathcal{O}^\lambda}(M_i,\Sigma X).
\end{align*}
Now $\Sigma X \in \mathcal{C}_i$ by Proposition \ref{connectedcomponents} and $Y \in \mathcal{C}_i$ by definition, while $M_i \in \mathcal{C}_i^{\hspace{1pt}\text{c}}$, so both of the flanking terms are $0$. This implies that the middle term, which is equal to $\stabHom_{\mathcal{O}^\lambda}(M_i,U')$, is also $0$. But this means that $U' = 0$, and hence $Z \oplus P \in \mathcal{C}_i$. Since, by definition, $\mathcal{C}_i$ is closed under direct summands, it follows that $Z \in \mathcal{C}_i$ as required. \\
\indent For the final claim, first notice that the above paragraph tells us that $\mathcal{C}_i$ is extension-closed. Moreover, given an object $C \in \mathcal{C}_i$, since $\MCM \mathcal{O}^\lambda$ is Frobenius we can always find conflations $C \to P \to Z$ and $X \to Q \to C$ with $X,Z \in \MCM \mathcal{O}^\lambda$ and $P,Q$ projective. Since projective $\mathcal{O}^\lambda$-modules are direct summands of sums of copies of $\mathcal{O}^\lambda$, we have $P,Q \in \mathcal{C}_i$ by definition, and then the previous paragraph tells us that $X,Z \in \mathcal{C}_i$. Therefore $\mathcal{C}_i$ is admissible and hence Frobenius.
\end{proof}

\indent This allows us to prove our main theorem:

\begin{thm} \label{triequiv}
Let $\widetilde{Q}$ and $Q_\lambda$ be as in Corollary \ref{addequiv3}. Then the $\Bbbk$-linear equivalence
\begin{align*}
\stabMCM \mathcal{O}^\lambda \simeq \bigoplus_{i=1}^r \proj \Pi(Q^{(i)}),
\end{align*}
of Corollary \ref{addequiv3} is a triangle equivalence, where the right hand side is a decomposition into triangulated subcategories satisfying $\proj \Pi(Q^{(i)}) \simeq \Dsg(R_{Q^{(i)}})$.
\end{thm}
\begin{proof}
By Lemma \ref{frobsubcat}, we know that $\mathcal{C}_i$ is a Frobenius subcategory of $\MCM \mathcal{O}^\lambda$. Using \cite[Theorem 3.15 (2)]{arentz}, it follows that $\mathcal{T}_i$ is equal to the stable category of the Frobenius category $\mathcal{C}_i$ for $1 \leqslant i \leqslant r$, and so the decomposition
\begin{align*}
\stabMCM \mathcal{O}^\lambda = \bigoplus_{i=1}^r \mathcal{T}_i
\end{align*}
is in fact a decomposition into triangulated subcategories. If we set $e^{(i)} = \sum_{j \in Q^{(i)}_0} e_j$, \cite[Proposition 2.3]{krause} implies that the functor
\begin{align*}
\stabHom_{\mathcal{O}^\lambda}(e^{(i)}\Pi^\lambda e_0,-) : \stabmod \mathcal{O}^\lambda \to \mod \stabEnd_{\mathcal{O}^\lambda}(e^{(i)}\Pi^\lambda e_0)
\end{align*}
restricts to a $\Bbbk$-linear equivalence $\mathcal{T}_i \simeq \proj \Pi(Q^{(i)})$. This equivalence also induces an algebraic triangulated structure on $\proj \Pi(Q^{(i)})$. Since this category is $\Bbbk$-linearly equivalent to $\Dsg(R_{Q^{(i)}})$, Theorem \ref{amiotthm} implies that they are triangle equivalent, completing the proof.
\end{proof}

\begin{rem}
In the first version of this paper, \cite{crawford}, a much longer argument was used to establish Theorem \ref{triequiv}. The original proof made use of the so-called knitting algorithm from \cite{iyamawemyss} to construct short exact sequences of $\mathcal{O}^\lambda$-modules, and explicitly listed many such sequences. The techniques used and the exact sequences given may be of independent interest.
\end{rem}


\section{A noncommutative geometric McKay correspondence} \label{inttheorysec}
We now look at a generalisation of the intersection theory of the minimal resolution of a Kleinian singularity $\Spec R_Q$ to a noncommutative setting. We begin by recalling the result in the commutative setting, whereby it is often referred to as the geometric McKay correspondence. \\
\indent Let $R_Q$ be a Kleinian singularity with corresponding extended Dynkin quiver $\widetilde{Q}$. Then the affine variety $\Spec R_Q$ is an isolated surface singularity which has a unique minimal resolution. It is well-known (see, for example, \cite[\S 6.4]{leu}) that the exceptional fibre of this minimal resolution is a union of $n$ irreducible curves $\gamma_i$, where $n$ is the number of vertices of $Q$. Moreover, each $\gamma_i$ is isomorphic to $\mathbb{P}^1$ and has self-intersection $-2$, and $\gamma_i \cap \gamma_j$ is either empty or a point. In fact, the dual graph of the exceptional fibre is given by the underlying graph of $Q$. Let $\Gamma$ be the $n \times n$ matrix with entries
\begin{gather*}
\Gamma_{ij} =
\left\{ 
\begin{array}{cl}
-2 & \text{if } i=j, \\
1 & \text{if } \gamma_i \text{ and } \gamma_j \text{ intersect,} \\
0 & \text{otherwise,}
\end{array} \right.
\end{gather*}
so that $\Gamma_{ij}$ is equal to the intersection multiplicity $\gamma_i \bullet \gamma_j$. With an appropriate labelling of the curves $\gamma_i$, we have $\Gamma = -C$ where $C$ is the Cartan matrix corresponding to the Dynkin type of $Q$; explicitly, $C = 2I-A$, where $A$ is the adjacency matrix of the underlying graph of $Q$. \\
\indent Now let $\widetilde{Q}$ be an extended Dynkin quiver with $n+1$ vertices, let $Q$ be the quiver obtained by removing the extending vertex, and let $\lambda = \varepsilon_0 = (1,0,\dots,0)$; that is, the weight at the extending vertex is $1$, and $0$ for all of the other vertices. We may then consider $\mathcal{O}^\lambda(\widetilde{Q})$ to be a noncommutative analogue of $R_Q$, the coordinate ring of the corresponding Kleinian singularity; indeed, these rings have equivalent singularity categories by Theorem \ref{introthm0}. We now seek to generalise the geometric McKay correspondence to $\mathcal{O}^\lambda(\widetilde{Q})$ for this particular choice of $\lambda$. To do so, we need an appropriate analogue of a resolution of the singular ring $\mathcal{O}^\lambda(\widetilde{Q})$; we use the definition given in \cite{qin}, where such a resolution is a noncommutative ring satisfying certain properties which we will recall below. The role of the exceptional curves in the resolution will be played by finite-dimensional simple modules, of which our resolution of $\mathcal{O}^\lambda(\widetilde{Q})$ has finitely many, and the intersection multiplicity of any two such modules is provided by \cite{mori}. In this section, we will prove the following: 

\begin{thm}[Theorem \ref{inttheory2}] \label{introthm2}
Let $\widetilde{Q}$ be an extended Dynkin quiver with corresponding Dynkin quiver $Q$ and let $\lambda = \varepsilon_0$. Then $\mathcal{O}^\lambda(\widetilde{Q})$ has a noncommutative resolution (which in fact can take the form $\mathcal{O}^\mu(\widetilde{Q})$ for some weight $\mu$ and so is a deformation), and the exceptional objects (the finite-dimensional simple modules) in this resolution may be indexed so that the corresponding intersection matrix is $-C$, where $C$ is the Cartan matrix corresponding to $Q$.
\end{thm}

\subsection{Noncommutative quasi-crepant resolutions}
We now give a precise definition of the noncommutative resolutions which appear in the above theorem, which is taken from \cite{qin}, and we prove a useful general result. We first recall a definition:

\begin{defn}
Let $A$ be a $\Bbbk$-algebra and $X,Y \in \Mod A$. We say that $X$ and $Y$ are \emph{$n$-isomorphic}, written $X \cong_{n} Y$, if there exists a third module $Z$ and homomorphisms
\begin{align*}
\phi : Z \to X \quad \text{and} \quad \psi : Z \to Y
\end{align*}
such that the kernels and cokernels of $\phi$ and $\psi$ each have GK dimension at most $n$.
\end{defn}

We can now give the precise definition of a \emph{noncommutative quasi-crepant resolution} from \cite[Definition 3.15]{qin}:

\begin{defn}
Let $A$ be a noetherian Auslander-Gorenstein $\Bbbk$-algebra with $\GKdim(A) = d$. Then a \emph{noncommutative quasi-crepant resolution} (NQCR) of $A$ is a triple $(B,M,N)$ where $B$ is a noetherian Auslander regular, Cohen-Macaulay $\Bbbk$-algebra with $\GKdim(B) = d$ and where ${}_B M_A$ and ${}_A N_{B}$ are finitely generated bimodules which are reflexive on both sides and which satisfy 
\begin{align*}
M \otimes_A N \cong_{d-2} B \quad \text{and} \quad N \otimes_B M \cong_{d-2} A.
\end{align*}
\end{defn}

NQCRs have the following useful property, which is not proven in \cite{qin} but provides a useful complement to \cite[Theorem 0.6]{qin}. We remark that the following statement remains true if one replaces ``noncommutative quasi-crepant resolution'' by ``noncommutative quasi-resolution'' (see \cite[Definition 3.2]{qin}), the latter notion being slightly weaker.

\begin{lem} \label{moritainvariant}
If $B$ is a noncommutative quasi-crepant resolution of $A$ and $C$ is Morita equivalent to $B$, then $C$ is also a noncommutative quasi-crepant resolution of $A$.
\end{lem}
\begin{proof}
Let $d = \GKdim A$. Since $B$ is a NQCR of $A$, it is Auslander regular and Cohen-Macaulay, has GK dimension $d$, and there exist bimodules ${}_B M_A$ and ${}_A N_{B}$ which are finitely generated and reflexive on both sides and which satisfy 
\begin{align*}
M \otimes_A N \cong_{d-2} B \quad \text{and} \quad N \otimes_B M \cong_{d-2} A.
\end{align*}
Moreover, since $C$ is Morita equivalent to $B$, there exists a progenerator $P \in \mod B$ with $C \cong \End(P_B)$, and we may view $P$ as a $(C,B)$-bimodule. 
\\
\indent First note that $C$ has GK dimension $d$ by standard Morita theory, and by applying \cite[Propositon 4.3]{yek}, we deduce that $C$ is also Auslander regular and Cohen-Macaulay. Writing ${}_B Q_C = \Hom({}_C P, C) \cong \Hom(P_B,B)$, define two bimodules ${}_C \widetilde{M}_A$ and ${}_A \widetilde{N}_{C}$ as follows:
\begin{align*}
\widetilde{M} = P \otimes_B M \quad \text{and} \quad \widetilde{N} = N \otimes_B Q.
\end{align*}

Since we have pairs of mutually inverse equivalences
\begin{gather*}
- \otimes_C Q : \mod B \to \mod C, \qquad - \otimes_B P : \mod C \to \mod B, \\
P \otimes_B - : B\text{-mod} \to C\text{-mod}, \qquad Q \otimes_C - : C\text{-mod} \to B\text{-mod},
\end{gather*}
and since $M$ and $N$ are reflexive on both sides, it follows that $\widetilde{M}$ and $\widetilde{N}$ are reflexive on both sides. \\
\indent It remains to show that $\widetilde{M} \otimes_A \widetilde{N} \cong_{d-2} C$ and $\widetilde{N} \otimes_C \widetilde{M} \cong_{d-2} A$. The latter of these is immediate, since
\begin{align*}
\widetilde{N} \otimes_C \widetilde{M} = N \otimes_B Q \otimes_C P \otimes_B M \cong N \otimes_B M \cong_{d-2} A,
\end{align*}
where the first isomorphism follows from \cite[18.17 Proposition]{lam} and the $(d-2)$-isomorphism follows since $(B,M,N)$ is a NQCR of $A$. We now wish to show that $\widetilde{M} \otimes_A \widetilde{N} \cong_{d-2} C$. Since $M \otimes_A N \cong_{d-2} B$, there exists a bimodule ${}_B Z_B$ and morphisms
\begin{align*}
\phi : Z \to M \otimes_A N \quad \text{and} \quad \psi : Z \to B
\end{align*}
such that the kernels and cokernels of $\phi$ and $\psi$ have GK dimension at most $d-2$. Define $\widetilde{Z} = P \otimes_B Z \otimes_B Q$ and morphisms
\begin{gather*}
\widetilde{\phi} : \widetilde{Z} \to \widetilde{M} \otimes_A \widetilde{N}, \qquad \widetilde{\phi} = \id_P \otimes \phi \otimes \id_Q, \\
\widetilde{\psi} : \widetilde{Z} \to C, \qquad \widetilde{\psi}(p \otimes x \otimes q) = q(p \psi(x)).
\end{gather*}
We claim that the kernels and cokernels of these maps have GK dimension at most $d-2$, which will complete the proof. \\
\indent We first consider the kernel and cokernel of $\widetilde{\phi}$. Since $P$ and $Q$ are projective (hence flat) on the right and left respectively, we can make an identification $\ker \widetilde{\phi} \cong P \otimes_B \ker \phi \otimes_B Q$. Since $P$ and $Q$ are finitely generated modules, \cite[Proposition 8.3.14]{mandr} implies that $\GKdim \ker \widetilde{\phi} \leqslant \GKdim \ker \phi \leqslant d-2$, as required. Flatness of $P$ and $Q$ also allows us to make identifications $\im \widetilde{\phi} \cong P \otimes_B \im \phi \otimes_B Q$ and
$\coker \widetilde{\phi} \cong P \otimes_B \coker \phi \otimes_B Q$, and so $\GKdim \coker \widetilde{\phi} \leqslant \GKdim \coker \phi \leqslant d-2$. \\
\indent We now turn our attention to $\widetilde{\psi}$. Observe that we can view $\widetilde{\psi}$ as the composition
\begin{align*}
\widetilde{\psi} : P \otimes_{B} Z \otimes_B Q \xrightarrow{\id_P \otimes \psi \otimes \id_Q} P \otimes_{B} B \otimes_B Q \xrightarrow{m \otimes \id_Q} P \otimes_B Q \xrightarrow{\mu} C
\end{align*}
where $m : P \otimes_B B \to P$ is the multiplication map and $\mu: P \otimes_B Q \to C$ is the evaluation map. Flatness of $Q$ implies that $m \otimes \id_Q$ is an isomorphism, while $\mu$ is an isomorphism because $P$ is a progenerator. Thus $\ker \widetilde{\psi} = \ker (\id_P \otimes \psi \otimes \id_Q)$, and again we can identify this with $P \otimes_B \ker \psi \otimes_B Q$, which has GK dimension at most $d-2$ using the same argument as in the previous paragraph. Similarly, we have an identification $\im \widetilde{\psi} \cong P \otimes_B \im \psi \otimes_B Q$, and arguing again as above, we find that $\im \widetilde{\psi}$ has GK dimension at most $d-2$, completing the proof.
\end{proof}

\subsection{Intersection theory for a family of noncommutative resolutions}
We return now to the $\Bbbk$-algebra of interest, namely $\mathcal{O}^\lambda$ where $\lambda = \varepsilon_0$. Our first aim is to identify an appropriate NQCR, which we have in fact already done:

\begin{lem} \label{nccr}
$\Pi^\lambda$ is a NQCR of $\mathcal{O}^\lambda$.
\end{lem}
\begin{proof}
Since $\mathcal{O}^\lambda$ is Auslander-Gorenstein, $\Pi^\lambda$ is Auslander regular and Cohen-Macaulay, and they both have GK dimension 2 \cite[Theorem 1.5, Theorem 1.6]{cbh}, it suffices to show that $\Pi^\lambda e_0 \in \text{bimod-} (\Pi^\lambda, \mathcal{O}^\lambda)$ and $e_0 \Pi^\lambda \in \text{bimod-} (\mathcal{O}^\lambda, \Pi^\lambda)$ are reflexive on both sides, and that
\begin{align*}
\Pi^\lambda e_0 \otimes_{\mathcal{O}^\lambda} e_0 \Pi^\lambda \cong_0 \Pi^\lambda \quad \text{and} \quad e_0 \Pi^\lambda \otimes_{\Pi^\lambda} \Pi^\lambda e_0 \cong_0 \mathcal{O}^\lambda.
\end{align*}
\indent By Lemma \ref{Pireflex}, $\Pi^\lambda e_0$ is reflexive as a right $\mathcal{O}^\lambda$-module, and since it is a generator for $\mod \mathcal{O}^\lambda$, \cite[Proposition 18.17]{lam} implies it is also reflexive as left $\Pi^\lambda$-module. Similarly, $e_0 \Pi^\lambda$ is a reflexive module on both sides. \\
\indent The 0-isomorphism $e_0 \Pi^\lambda \otimes_{\Pi^\lambda} \Pi^\lambda e_0 \cong_0 \mathcal{O}^\lambda$ follows from the fact that these two modules are actually isomorphic. To see that $\Pi^\lambda e_0 \otimes_{\mathcal{O}^\lambda} e_0 \Pi^\lambda \cong_0 \Pi^\lambda$, it suffices to show that the multiplication map
\begin{align*}
m : \Pi^\lambda e_0 \otimes_{\mathcal{O}^\lambda} e_0 \Pi^\lambda \to \Pi^\lambda
\end{align*}
has finite-dimensional kernel and cokernel. The cokernel of $m$ is $\Pi^\lambda/\Pi^\lambda e_0 \Pi^\lambda \cong \Pi(Q_\lambda) = \Pi(Q)$, which is finite-dimensional. To see that $K = \ker m$ is finite-dimensional, factor $m$ as
\begin{align*}
m : \Pi^\lambda e_0 \otimes_{\mathcal{O}^\lambda} e_0 \Pi^\lambda \xrightarrow{\pi} \Pi^\lambda e_0 \Pi^\lambda \rightarrow \Pi^\lambda
\end{align*}
where $\pi$ is surjective and hence $\ker \pi = K$. We then have a short exact sequence
\begin{align*}
0 \to K \to \Pi^\lambda e_0 \otimes_{\mathcal{O}^\lambda} e_0 \Pi^\lambda \xrightarrow{\pi} \Pi^\lambda e_0 \Pi^\lambda \to 0.
\end{align*}
Since $\Pi^\lambda e_0$ is a finitely generated $\mathcal{O}^\lambda$-module, $\Pi^\lambda e_0 \otimes_{\mathcal{O}^\lambda} e_0 \Pi^\lambda$ is a finitely generated $\Pi^\lambda$-module, and so $K$ is also a finitely generated $\Pi^\lambda$-module. Applying $- \otimes_{\Pi^\lambda} \Pi^\lambda e_0$ to the above sequence, which is an exact functor since $\Pi^\lambda e_0$ is a projective left $\Pi^\lambda$-module, we obtain the short exact sequence
\begin{align*}
0 \to K \otimes_{\Pi^\lambda} \Pi^\lambda e_0 \to \Pi^\lambda e_0 \otimes_{\mathcal{O}^\lambda} e_0 \Pi^\lambda \otimes_{\Pi^\lambda} \Pi^\lambda e_0 \to \Pi^\lambda e_0 \Pi^\lambda \otimes_{\Pi^\lambda} \Pi^\lambda e_0 \to 0.
\end{align*}
It is easy to see that the above sequence is in fact
\begin{align*}
0 \to K e_0 \to \Pi^\lambda e_0 \xrightarrow{\sim} \Pi^\lambda e_0 \to 0,
\end{align*}
and so $K e_0 = 0$. It follows that $K$ has the structure of a finitely generated right $ \Pi^\lambda / \Pi^\lambda e_0 \Pi^\lambda$-module, and is therefore finite-dimensional. By definition, we find that $\Pi^\lambda$ is a NQCR of $\mathcal{O}^\lambda$.
\end{proof}

We can actually obtain infinitely many noncommutative resolutions of $\mathcal{O}^\lambda$ using the dual reflections $r_i$ of \cite{cbh}, the definition of which was given prior to Lemma \ref{bodlemma}. It is clear that the $r_i$ preserve the $\mathbb{Z}^{n+1}$ lattice inside $\Bbbk^{n+1}$. It was also noted earlier that $\lambda \cdot \delta = r_i \lambda \cdot \delta$ for all $\lambda \in \Bbbk^{\widetilde{Q}_0}$ and $i \in \widetilde{Q}_0$, so that the $r_i$ preserve the affine hyperplanes $\{ \lambda \in \Bbbk^{n+1} \mid \lambda \cdot \delta = c \}$ for each $c \in \Bbbk$; since $\varepsilon_0 \cdot \delta = 1$, we are primarily interested in the case $c=1$. Crawley-Boevey--Holland proved the following useful result:

\begin{lem}[{{\cite[Corollary 5.2]{cbh}}}] \label{refthm}
Let $\rho$ be a composition of dual reflections. Then $\Pi^\lambda$ is Morita equivalent to $\Pi^{\rho(\lambda)}$.
\end{lem}

\indent By combining Lemmas \ref{moritainvariant}, \ref{nccr} and \ref{refthm}, we obtain the following:

\begin{cor} \label{resolutioncor}
$\Pi^{\rho(\lambda)}$ is a NQCR of $\mathcal{O}^\lambda$ for any composition of dual reflections $\rho$.
\end{cor}

\indent As stated previously, we need to identify an analogue of the exceptional curves appearing in the minimal resolution of a Kleinian singularity. When $\lambda = \varepsilon_0$, by Lemma \ref{PiofDynkin}, $\Pi^\lambda$ has precisely $n$ isoclasses of finite-dimensional simple modules, and hence by Morita equivalence so does $\Pi^{\rho(\lambda)}$ for any composition of dual reflections $\rho$. These will play the role of the exceptional objects in our noncommutative resolution. \\
\indent We also require a notion of intersection multiplicity for the exceptional objects, which is provided by \cite{mori}. Given a nonsingular noetherian ring $S$ and $M,N \in \mod S$ which satisfy $\dim_\Bbbk \Ext_S^\ell(M,N) < \infty$ for all $\ell \geqslant 0$, we define the \emph{intersection multiplicity} of $M$ and $N$ to be 
\begin{gather*}
M \bullet N \coloneqq \sum_{\ell \geqslant 0} (-1)^{\ell+1} \dim_\Bbbk \Ext_{S}^{\ell}(M,N)
\end{gather*}
(note that this sum has finitely many terms since $S$ is nonsingular).\\
\indent We are now in a position to prove a preliminary version of Theorem \ref{introthm2}:

\begin{thm}\label{inttheory}
Let $\widetilde{Q}$ be an extended Dynkin quiver with $n+1$ vertices, and let $\lambda = \varepsilon_0$. Let $\mu = \rho(\lambda)$, where $\rho$ is any composition of dual reflections, so that $\Pi^\mu$ is a NQCR of $\mathcal{O}^\lambda$. Then $\Pi^\mu$ has precisely $n$ finite-dimensional simple modules $S_i$ up to isomorphism, and with a suitable indexing of them, the intersection matrix $\Gamma$ with entries $\Gamma_{ij} = S_i \bullet S_j$ is $-C$, where $C$ is the Cartan matrix corresponding to $Q$.
\end{thm}
\begin{proof}
The discussion after Corollary \ref{resolutioncor} shows that $\Pi^\mu$ has $n$ finite-dimensional simple modules $S_i$ up to isomorphism, so it remains to prove the result on the intersection multiplicities. \\
\indent Since Morita equivalence preserves dimensions of Hom and Ext groups, we are able to calculate the intersection numbers of the finite-dimensional $\Pi^\mu$-modules by performing the calculations over $\Pi^\lambda$ instead. Identifying $\Pi^\lambda$-modules with representations of $\overline{\widetilde{Q}}$ which satisfy the relations coming from $\Pi^\lambda$, \cite[Lemma 7.2 (6), Theorem 7.4]{cbh} tells us that the dimension vector of $S_i$ is $\varepsilon_i \in \mathbb{N}^{n+1}$. It follows that
\begin{gather}
S_i \cong \frac{e_i \Pi^\lambda}{\bigoplus_{\genfrac{}{}{0pt}{}{\alpha \in \overline{\widetilde{Q}}_1}{t(\alpha)=i}} \alpha \Pi^\lambda}. \label{simplemod}
\end{gather}
Also observe that
\begin{gather}
\Hom_{\Pi^\lambda}(e_i \Pi^\lambda,S_j) = 
\left\{ 
\begin{array}{cl}
\Bbbk e_i & \text{if } i = j, \\
0 & \text{if } i \neq j.
\end{array} \right. \label{homsimple}
\end{gather}
The proof of \cite[Lemma 10.1]{cbh} shows that, for each $i \neq 0$, there is an exact sequence of $\Pi^\lambda$-modules
\begin{gather*}
0 \to e_i \Pi^\lambda \xrightarrow{\phi} \bigoplus_{k \in \partial i} e_k \Pi^\lambda \xrightarrow{\psi} e_i \Pi^\lambda \to S_i \to 0.
\end{gather*}
Since the modules $e_k \Pi^\lambda$ are direct summands of $\Pi^\lambda$ and hence projective, this is in fact a projective resolution of $S_i$. Now let $1 \leqslant j \leqslant n$. Seeking to calculate the extension groups between $S_i$ and $S_j$, we apply $\Hom_{\Pi^\lambda}(-,S_j)$ to the corresponding deleted resolution to obtain the complex
\begin{gather}
0 \to \Hom_{\Pi^\lambda}(e_i \Pi^\lambda,S_j) \to \bigoplus_{k \in \partial i} \Hom_{\Pi^\lambda}(e_k \Pi^\lambda,S_j) \to \Hom_{\Pi^\lambda}(e_i \Pi^\lambda,S_j) \to 0. \label{intersectioncomplex}
\end{gather}
We now consider three distinct cases when computing the homology of this complex. If $j=i$ then, using (\ref{homsimple}), as a complex of vector spaces (\ref{intersectioncomplex}) becomes
\begin{gather*}
0 \to \Bbbk \to 0 \to \Bbbk \to 0
\end{gather*}
and so we can immediately read off that
\begin{gather*}
\dim_\Bbbk \Hom_{\Pi^\lambda}(S_i,S_i) = 1 = \dim_\Bbbk \Ext_{\Pi^\lambda}^2(S_i,S_i), \\ 
\dim_\Bbbk \Ext_{\Pi^\lambda}^\ell (S_i,S_i) = 0 \quad \text{for } \ell = 1 \text{ or } \ell \geqslant 3 ,
\end{gather*}
and so $S_i \bullet S_i = -1+0-1 = -2$. If $j \in \partial i$, then (\ref{intersectioncomplex}) becomes
\begin{gather*}
0 \to 0 \to \Bbbk \to 0 \to 0
\end{gather*}
and so
\begin{gather*}
\dim_\Bbbk \Ext_{\Pi^\lambda}^1(S_i,S_i) = 1, \qquad \dim_\Bbbk \Ext_{\Pi^\lambda}^\ell (S_i,S_i) = 0 \quad \text{for } \ell = 0 \text{ or } \ell \geqslant 2.
\end{gather*}
That is, if $i$ and $j$ are adjacent in $\widetilde{Q}$, then $S_i \bullet S_j = 0+1+0 =1$. Finally, if $j \neq i$ and $j \notin \partial i$ then (\ref{intersectioncomplex}) becomes
\begin{gather*}
0 \to 0 \to 0 \to 0 \to 0
\end{gather*} and clearly
\begin{gather*}
\dim_\Bbbk \Ext_{\Pi^\lambda}^\ell (S_i,S_i) = 0 \quad \text{for } \ell \geqslant 0,
\end{gather*}
and so $S_i \bullet S_j = 0$ in this case. It follows that the intersection matrix $\Gamma$ satisfies $\Gamma = -C$.
\end{proof}

\indent The above result should be seen as a noncommutative analogue of the geometric McKay correspondence. However, we can strengthen this result by showing that $\mathcal{O}^\lambda$ possesses a NQCR which is actually a ``deformation'': that is, a NQCR of the form $\mathcal{O}^\mu$ for some weight $\mu$. Since we are restricting our attention to quasi-dominant weights, the fact that $\mathcal{O}^\mu$ is nonsingular forces $\mu_i \succ 0$ for all $i \geqslant 1$ (see Lemma \ref{quasidomlem}). It is not immediately clear that such a deformation exists; we prove its existence in the next subsection. 

\subsection{$\mathcal{O}^\lambda$ has a NQCR which is a deformation}
\indent The dual reflections defined earlier also appear in the so-called \emph{numbers game} of \cite{mozes}. The relationship between this game and our setting is that the moves considered by Mozes can equivalently be described as an application of a dual reflection to a weight $\lambda$. This allows us to make use of some of the results from this paper; in particular we are able to prove that, for $\lambda = \varepsilon_0$, NQCRs of $\mathcal{O}^\lambda$ which are also deformations exist:

\begin{lem} \label{reflecttosmooth}
Let $\widetilde{Q}$ be an extended Dynkin quiver with $n+1$ vertices. Then there exists a sequence of dual reflections $\rho$ such that $\rho(\varepsilon_0)_i > 0$ for all $i \neq 0$; in particular, $\rho(\varepsilon_0)$ is quasi-dominant.
\end{lem}
\begin{proof}
It suffices to show that we can find such a sequence of dual reflections when we work over the field $\mathbb{R}$, since any such sequence will also have the desired effect on $\varepsilon_0$ when we work over our algebraically closed field $\Bbbk$ of characteristic $0$. Accordingly, we write $<$ instead of $\prec$ for the total order. Write $G$ for the group generated by the dual reflections, which is simply the Weyl group of type corresponding to $Q$. Lemma 5.5 of \cite{mozes}, when translated into our notation, says that $\{ \lambda \in \mathbb{R}^{n+1} \mid \lambda_i \geqslant 0 \text{ for all } 0 \leqslant i \leqslant n \}$ is a fundamental domain for the action of $G$ on $\{ \lambda \in \mathbb{R}^{n+1} \mid \lambda \cdot \delta > 0 \}$. Recalling that $G$ preserves the affine hyperplane $V \coloneqq \{\lambda \in \mathbb{R}^{n+1} \mid \lambda \cdot \delta = 1 \}$, it follows that $V = \bigcup_{\rho \in G} \rho \gap U$, where $U$ is the $n$-simplex $\{ \lambda \in \mathbb{R}^{n+1} \mid \lambda_i \geqslant 0 \text{ for all } 0 \leqslant i \leqslant n \text{ and } \lambda \cdot \delta = 1\}$. Let $H = \{ \lambda \in V \mid \lambda_i > 0 \text{ for all } i \neq 0 \}$, which is a convex subset of $V$ containing open balls of arbitrarily large diameter. Since each $\rho U$ has the same finite diameter, there exists some $\rho \in G$ with $\rho U \subseteq H$. In particular, $\rho (\varepsilon_0) \in H$; that is, $\rho(\varepsilon_0)_i > 0$ for all $i \neq 0$.
\end{proof}

\begin{rem}
By playing Mozes' numbers game, one can often determine an explicit sequence of dual reflections $\rho$ satisfying the hypotheses of Lemma \ref{reflecttosmooth}. For example, if $\widetilde{Q} = \widetilde{\mathbb{A}}_4$, then the numbers game starting with the initial configuration $(-3,1,1,1,1)$ terminates at $\varepsilon_0$, and so by applying the corresponding dual reflections in reverse we obtain the desired $\rho$. More generally, \cite[Proposition 5.1]{schedler} tells us that when $\widetilde{Q}$ is of type $\widetilde{\mathbb{A}}_{2m}$, $\widetilde{\mathbb{D}}_{4m}$, $\widetilde{\mathbb{D}}_{4m+1}$, $\widetilde{\mathbb{E}}_{6}$ or $\widetilde{\mathbb{E}}_{8}$, where $m$ is a positive integer, then the numbers game starting with the initial configuration $(1- \sum_{i=1}^{n} \delta_i, 1, 1, \dots, 1)$ terminates at $\varepsilon_0$, and so this determines a sequence of dual reflections $\rho$ such that $\rho(\lambda)_i  > 0$ for all $i \neq 0$.
\end{rem}

\indent We are now in a position to prove Theorem \ref{introthm2}:

\begin{thm}\label{inttheory2}
Let $\widetilde{Q}$ be an extended Dynkin quiver with $n+1$ vertices, and let $\lambda = \varepsilon_0$. Then $\mathcal{O}^\lambda$ has a NQCR of the form $\mathcal{O}^\mu$, where $\mathcal{O}^\mu$ has precisely $n$ finite-dimensional simple modules $S_i$ up to isomorphism. With a suitable indexing of the $S_i$, the intersection matrix $\Gamma$ with entries $\Gamma_{ij} = S_i \bullet S_j$ is $-C$, where $C$ is the Cartan matrix corresponding to $Q$.
\end{thm}
\begin{proof}
Lemma \ref{reflecttosmooth} tells us that there exists a sequence of dual reflections $\rho$ such that $\mathcal{O}^\mu$ is nonsingular, where $\mu = \rho(\lambda)$. Since $\Pi^\lambda$ is a resolution of $\mathcal{O}^\lambda$ and there are Morita equivalences between $\Pi^\lambda$, $\Pi^\mu$, and $\mathcal{O}^\mu$ (by \cite[Corollary 5.2, Corollary 9.6]{cbh}), it follows that $\mathcal{O}^\mu$ is a NQCR of $\mathcal{O}^\lambda$. Finally, these Morita equivalences combined with Theorem \ref{inttheory} tells us that $\mathcal{O}^\mu$ has precisely $n$ finite-dimensional simple modules $S_i$ up to isomorphism, and since Morita equivalences preserve dimensions of Hom and Ext groups, the claimed intersection multiplicities follow from Theorem \ref{inttheory} as well.
\end{proof}


\appendix

\section{Uniqueness of the translation functor on objects of $\proj \Pi(Q)$ when $Q$ is Dynkin}

In this appendix, we show that if $Q$ is Dynkin and $\proj \Pi(Q)$ has the structure of a (not necessarily algebraic) triangulated category, then the translation functor $\Sigma$ is uniquely determined on objects of $\proj \Pi(Q)$. In particular, this tells us how the translation functor acts on objects in Theorem \ref{triequiv}. Although this is known from that result, we believe that an elementary and relatively short proof of this fact may be of independent interest. \\
\indent For the remainder of this section, write  $P_1, \dots, P_n$ for the $n$ indecomposable projective right $\Pi(Q)$-modules corresponding to the vertices of $Q$. Write $W_0, \dots, W_n$ for the $n+1$ irreducible representations of the finite group $G$ corresponding to $Q$. Since $\proj \Pi(Q)$ is Krull-Schmidt, it is easy to see that $\Sigma P_i = P_j$ for some $j$, so write $\sigma$ for the permutation of the vertices of $Q$ satisfying $\Sigma P_i = P_{\sigma(i)}$. The map $W_i \to W_i^*$ sending a representation to its dual is an involution of $\{W_1, \dots, W_n\}$ (where we intentionally omit $W_0$), and we can view this map as an automorphism $\nu$ of $Q$. Throughout this section, all Hom spaces are over $\Pi(Q)$, and we omit this subscript. The aim of this section is to prove the following result, which we achieve by analysing cases.

\begin{thm} \label{graphautos}
Consider the category $\proj \Pi(Q)$ with some triangulated structure with translation functor $\Sigma$. Then $\sigma = \nu$ as automorphisms of $Q$.
\end{thm}

\begin{rem}
Explicitly, $\nu$ is the identity automorphism of $Q$ when $Q$ is $\mathbb{A}_1, \mathbb{D}_{n}$ ($n$ even), $\mathbb{E}_7$, or $\mathbb{E}_8$, and it is the unique graph automorphism of order $2$ when $Q$ is $\mathbb{A}_n$ ($n \geqslant 2$), $\mathbb{D}_n$ ($n$ odd), or $\mathbb{E}_6$.
\end{rem}

\indent Adapting the proof of Lemma \ref{graphautolem}, $\sigma$ is necessarily a graph automorphism of $Q$. Therefore, since the automorphism group of an $\mathbb{A}_1$, $\mathbb{E}_7$, or $\mathbb{E}_8$ graph is trivial, it immediately follows that $\sigma$ is the identity in these cases. For the remaining cases, we argue that $\Sigma$ is uniquely determined using Lemma \ref{graphautolem} and by considering the dimensions of the Hom spaces between the $P_i$. We record the dimensions of these Hom spaces in the following lemma; the $\mathbb{E}_7$ and $\mathbb{E}_8$ cases are unnecessary and hence omitted, but they can be established in the same way.

\begin{lem} \label{homspacedims}
Let $Q$ be a Dynkin quiver with $n$ vertices and let $P_1, \dots, P_n$ be the $n$ indecomposable projective right $\Pi(Q)$-modules corresponding to the vertices of $Q$. Let $H(Q)$ be the matrix with \begin{align*}
H(Q)_{ij} = \dim_\Bbbk \Hom(P_j,P_i) = \dim_\Bbbk e_i \Pi(Q) e_j.
\end{align*}
\begin{enumerate}[{\normalfont (1)},leftmargin=*,topsep=0pt,itemsep=0pt]
\item If $Q = \mathbb{A}_n$ then
\begin{align*}
H(\mathbb{A}_n) = \begin{pmatrix}
1 & 1 & 1 & \cdots & 1 & 1 & 1 \\ 
1 & 2 & 2 & \cdots & 2 & 2 & 1 \\ 
1 & 2 & 3 & \cdots & 3 & 2 & 1 \\ 
\vdots & \vdots & \vdots & \ddots & \vdots & \vdots & \vdots \\
1 & 2 & 3 & \cdots & 3 & 2 & 1 \\ 
1 & 2 & 2 & \cdots & 2 & 2 & 1 \\ 
1 & 1 & 1 & \cdots & 1 & 1 & 1 
\end{pmatrix}.
\end{align*}
\item If $Q = \mathbb{D}_n$ then
\begin{align*}
H(\mathbb{D}_n) = \begin{pmatrix}
2 & 2 & 2 & \cdots & 2 & 1 & 1 \\ 
2 & 4 & 4 & \cdots & 4 & 2 & 2 \\ 
2 & 4 & 6 & \cdots & 6 & 3 & 3 \\ 
\vdots & \vdots & \vdots & \ddots & \vdots & \vdots & \vdots \\
2 & 4 & 6 & \cdots & 2(n-2) & n-2 & n-2 \\ 
1 & 2 & 3 & \cdots & n-2 & \ceil*{\tfrac{n-1}{2}} & \floor*{\tfrac{n-1}{2}} \\[2pt]
1 & 2 & 3 & \cdots & n-2 & \floor*{\tfrac{n-1}{2}} &  \ceil*{\tfrac{n-1}{2}} 
\end{pmatrix}.
\end{align*}
\item If $Q = \mathbb{E}_6$ then
\begin{gather*}
H(\mathbb{E}_6) = 
\begin{pmatrix}
4 & 2 & 4 & 6 & 4 & 2 \\
2 & 2 & 3 & 4 & 3 & 2 \\
4 & 3 & 6 & 8 & 6 & 3 \\
6 & 4 & 8 & 12 & 8 & 4 \\
4 & 3 & 6 & 8 & 6 & 3 \\
2 & 2 & 3 & 4 & 3 & 2
\end{pmatrix}.
\end{gather*}
\end{enumerate}
\end{lem}
\begin{proof}
These can be calculated using \cite[\S 4]{erdmannA}, \cite[3.4]{erdmannD}, and \cite[Theorem 2.3.b]{malkin}.
\end{proof}

We now begin our case-by-case argument. In each case, the technique is the same: seeking a contradiction, we show that if $\sigma$ is a graph automorphism of $Q$ different from the one given in Proposition \ref{graphautos} then we arrive at a contradiction. We begin with the type $\mathbb{A}$ case:

\begin{prop} \label{AnProp}
Let $\sigma$ be the graph automorphism of $\mathbb{A}_n$ induced by the translation functor $\Sigma$ on $\proj \Pi(\mathbb{A}_n)$. Then $\sigma$ is the identity when $n = 1$, and it is the unique order 2 graph automorphism when $n \geqslant 2$.
\end{prop}
\begin{proof}
We have already established the $n=1$ case, so suppose $n \geqslant 2$. By Lemma \ref{graphautolem}, $\sigma$ is either the identity or has order $2$ so, seeking a contradiction, suppose that $\sigma$ is the identity; that is $\Sigma P_i = P_i$ for all $i$. Consider the nonzero morphism $P_1 \to P_n$ given by left multiplication by $\overline{\alpha}_{n-1} \overline{\alpha}_{n-2} \dots \overline{\alpha}_{1}$, which gives rise to a distinguished triangle
\begin{align*}
P_1 \to P_n \to M \to P_1
\end{align*}
for some $M \in \proj \Pi(\mathbb{A}_n)$. Applying $\Hom(-,P_n)$, this gives rise to an exact sequence
\begin{displaymath}
\begin{tikzcd}[column sep=small,>=stealth,matrix scale=1, transform shape, nodes={scale=1}] 
\Hom(P_n,P_n)  \arrow{r}[shift={(0pt,4pt)}]{\cdot \overline{\alpha}_{n-1} \overline{\alpha}_{n-2} \dots \overline{\alpha}_{1}} \arrow[equal]{d} & \Hom(P_1,P_n) \arrow{r}[shift={(0pt,4pt)}]{\beta} \arrow[equal]{d} & \Hom(M,P_n)  \arrow{r}[shift={(0pt,4pt)}]{\gamma} & \Hom(P_n,P_n) \arrow{r}[shift={(0pt,4pt)}]{\cdot \overline{\alpha}_{n-1} \overline{\alpha}_{n-2} \dots \overline{\alpha}_{1}} \arrow[equal]{d} & \Hom(P_1,P_n)  \arrow[equal]{d} \\
\Bbbk e_n & \Bbbk \overline{\alpha}_{n-1} \overline{\alpha}_{n-2} \dots \overline{\alpha}_{1} & {} & \Bbbk e_n & \Bbbk \overline{\alpha}_{n-1} \overline{\alpha}_{n-2} \dots \overline{\alpha}_{1}
\end{tikzcd}
\end{displaymath}
\noindent where we use Lemma \ref{homspacedims} to write down bases for each of the Hom spaces. Now the left hand map is surjective, so exactness implies that $\beta$ is the zero map, which forces $\gamma$ to be injective. Moreover, the right hand map is injective, so that $\gamma$ is the zero map. In particular, Lemma \ref{homspacedims} implies that we have $\Hom(M,P_n) = 0$ and so $M = 0$, but this tells us that $P_1 \cong P_n$ which is absurd. Therefore $\sigma$ must be the unique order 2 graph automorphism of $\mathbb{A}_n$.
\end{proof}

\indent We now turn our attention to the type $\mathbb{E}$ cases:

\begin{prop} \label{EnProp}
Let $\sigma$ be the graph automorphism of $\mathbb{E}_n$ induced by the translation functor $\Sigma$ on $\proj \Pi(\mathbb{E}_n)$, where $n \in \{ 6,7,8 \}$. Then $\sigma$ is the identity when $n \neq 6$, and it is the unique order 2 graph automorphism when $n = 6$.
\end{prop}
\begin{proof}
Again, the $\mathbb{E}_7$ and $\mathbb{E}_8$ cases are immediate from Lemma \ref{graphautolem}, so consider $\mathbb{E}_6$. By Lemma \ref{graphautolem}, $\sigma$ is either the identity or has order $2$ so, seeking a contradiction, suppose that $\sigma$ is the identity. Consider the nonzero morphism $P_2 \to P_6$ given by left multiplication by $\alpha_5 \overline{\alpha}_{4} \alpha_3 \overline{\alpha}_{2}$, which gives rise to a distinguished triangle
\begin{align*}
P_2 \to P_6 \to M \to P_2
\end{align*}
for some $M \in \proj \Pi(\mathbb{E}_6)$. Applying $\Hom(-,P_6)$, this gives rise to an exact sequence
\begin{displaymath}
\begin{tikzcd}[column sep=small,>=stealth,matrix scale=1, transform shape, nodes={scale=1}] 
\Hom(P_6,P_6)  \arrow{r}[shift={(0pt,4pt)}]{\cdot \alpha_5 \overline{\alpha}_{4} \alpha_3 \overline{\alpha}_{2}} \arrow[equal]{d} & \Hom(P_2,P_6) \arrow{r}[shift={(0pt,4pt)}]{\beta} \arrow[equal]{d} & \Hom(M,P_6) \arrow{r}[shift={(0pt,4pt)}]{\gamma} & \Hom(P_6,P_6) \arrow{r}[shift={(0pt,4pt)}]{\cdot \alpha_5 \overline{\alpha}_{4} \alpha_3 \overline{\alpha}_{2} } \arrow[equal]{d} & \Hom(P_2,P_6)  \arrow[equal]{d} \\
\substack{ \Bbbk e_6 \\ \oplus \\ \Bbbk \alpha_5 \overline{\alpha}_4 \alpha_1 \overline{\alpha}_1 \alpha_4 \overline{\alpha}_5}
& \hspace{-10pt} \substack{ \Bbbk \alpha_5 \overline{\alpha}_{4} \alpha_3 \overline{\alpha}_{2} \\ \oplus \\ \Bbbk \alpha_5 \overline{\alpha}_4 \alpha_1 \overline{\alpha}_1 \alpha_4 \overline{\alpha}_5 \alpha_5 \overline{\alpha}_{4} \alpha_3 \overline{\alpha}_{2}} \hspace{-10pt}
& {}
& \hspace{-10pt} \substack{ \Bbbk e_6 \\ \oplus \\ \Bbbk \alpha_5 \overline{\alpha}_4 \alpha_1 \overline{\alpha}_1 \alpha_4 \overline{\alpha}_5} \hspace{-10pt} 
& \hspace{-10pt} \substack{ \Bbbk \alpha_5 \overline{\alpha}_{4} \alpha_3 \overline{\alpha}_{2} \\ \oplus \\ \Bbbk \alpha_5 \overline{\alpha}_4 \alpha_1 \overline{\alpha}_1 \alpha_4 \overline{\alpha}_5 \alpha_5 \overline{\alpha}_{4} \alpha_3 \overline{\alpha}_{2}}
\end{tikzcd}
\end{displaymath}
\noindent where again we use Lemma \ref{homspacedims} to write down bases for each of the Hom spaces. We see that the left hand map is surjective and so $\beta$ is the zero map, and exactness implies that $\gamma$ is injective. Since the right hand map is injective it follows that $\gamma$ is the zero map. Therefore by  Lemma \ref{homspacedims} $\Hom(M,P_6) = 0$ and so $M = 0$, but this tells us that $P_2 \cong P_6$ which is absurd. Therefore $\sigma$ must be the unique order 2 graph automorphism of $\mathbb{E}_6$.
\end{proof}

\indent Finally we consider the type $\mathbb{D}$ cases. Since we claim that $\Sigma$ behaves differently depending on whether $n$ is odd or even, we have to consider these two cases separately; additionally, we consider the $n=4$ case separately since $\text{Aut}(\mathbb{D}_4) \cong S_3$ instead of it having order 2.

\begin{prop}  \label{D4Prop}
Let $\sigma$ be the graph automorphism of $\mathbb{D}_4$ induced by the translation functor $\Sigma$ on $\proj \Pi(\mathbb{D}_4)$. Then $\sigma$ is the identity.
\end{prop}
\begin{proof}
By Lemma \ref{graphautolem}, $\sigma$ is either the identity, a two-cycle which swaps a pair of vertices $i \neq 2 \neq j$, or it cycles the vertices $1,3,4$. We rule out the latter two possibilities. \\
\indent First suppose that $\sigma$ is a two-cycle: without loss of generality, $\sigma = (3 \hspace{2pt} 4)$. Consider the nonzero morphism $P_3 \to P_4$ given by left multiplication by $\overline{\alpha}_4 \alpha_3$, which gives rise to a distinguished triangle
\begin{align*}
P_3 \to P_4 \to M \to P_4
\end{align*}
for some $M \in \proj \Pi(\mathbb{D}_4)$. Applying $\Hom(-, P_3)$, this gives rise to an exact sequence
\begin{displaymath}
\begin{tikzcd}[column sep=small,>=stealth,matrix scale=1, transform shape, nodes={scale=1}] 
\Hom(P_3,P_3)  \arrow{r}[shift={(0pt,4pt)}]{\cdot \overline{\alpha}_3 \alpha_4} \arrow[equal]{d} & \Hom(P_4,P_3) \arrow{r}[shift={(0pt,4pt)}]{\beta} \arrow[equal]{d} & \Hom(M,P_3) \arrow{r}[shift={(0pt,4pt)}]{\gamma} & \Hom(P_4,P_3) \arrow{r}[shift={(0pt,4pt)}]{\cdot \overline{\alpha}_4 \alpha_3 } \arrow[equal]{d} & \Hom(P_3,P_3)  \arrow[equal]{d} \\
\Bbbk e_3 \oplus \Bbbk \overline{\alpha}_3 \alpha_4 \overline{\alpha}_4 \alpha_3 & \Bbbk \overline{\alpha}_3 \alpha_4 & {} & \Bbbk \overline{\alpha}_3 \alpha_4 & \Bbbk e_3 \oplus \Bbbk \overline{\alpha}_3 \alpha_4 \overline{\alpha}_4 \alpha_3
\end{tikzcd}
\end{displaymath}
\noindent Clearly the left hand map surjects, so exactness forces $\beta$ to be the zero map, which in turn implies that $\gamma$ is injective. The right hand map is injective, and exactness forces $\gamma$ to be the zero map. In particular we have $\Hom(M,P_3) = 0$ and so $M = 0$, but this tells us that $P_3 \cong P_4$ which is absurd. Therefore $\sigma$ is not a two-cycle. \\
\indent Now suppose that $\sigma$ is a three-cycle: without loss of generality,  $\sigma = (1 \hspace{2pt} 3 \hspace{2pt} 4)$. We now consider the triangle obtained from the morphism $\alpha_1 \alpha_3 \cdot : P_3 \to P_1$,
\begin{align*}
P_3 \to P_1 \to M \to P_4
\end{align*}
and seek to obtain contradiction. Applying the functor $\Hom(-,P_3)$, we get exactness of the following sequence:
\begin{displaymath}
\begin{tikzcd}[column sep=small,>=stealth,matrix scale=1, transform shape, nodes={scale=1}] 
\Hom(P_3,P_3)  \arrow{r}[shift={(0pt,4pt)}]{\cdot \overline{\alpha}_3 \alpha_4} \arrow[equal]{d} & \Hom(P_4,P_3) \arrow{r}[shift={(0pt,4pt)}]{\beta} \arrow[equal]{d} & \Hom(M,P_3) \arrow{r}[shift={(0pt,4pt)}]{\gamma} & \Hom(P_1,P_3) \arrow{r}[shift={(0pt,4pt)}]{\cdot \alpha_1 \alpha_3 }  \arrow[equal]{d} & \Hom(P_3,P_3)  \arrow[equal]{d} \\
\Bbbk e_3 \oplus \Bbbk \overline{\alpha}_3 \overline{\alpha}_1 \alpha_1 \alpha_3 & \Bbbk \overline{\alpha}_3 \alpha_4 & {} & \Bbbk \overline{\alpha}_3 \overline{\alpha}_1 & \Bbbk e_3 \oplus \Bbbk \overline{\alpha}_3 \overline{\alpha}_1 \alpha_1 \alpha_3 
\end{tikzcd}
\end{displaymath}
\noindent Again the left hand map is surjective, forcing $\beta$ to be the zero map and hence $\gamma$ to be injective. Moreover, the right hand map is injective, and so $\gamma$ must be the zero map. In particular we have $\Hom(M,P_3) = 0$ and so $M = 0$, but this tells us that $P_1 \cong P_3$ which is absurd. Therefore $\sigma$ is not a three-cycle, and hence must be the identity.
\end{proof}

\begin{prop}  \label{DnOddProp}
Let $n \geqslant 5$ be odd and let $\sigma$ be the graph automorphism of $\mathbb{D}_n$ induced by the translation functor $\Sigma$ on $\proj \Pi(\mathbb{D}_n)$. Then $\sigma$ is the unique graph automorphism of order 2.
\end{prop}
\begin{proof}
By Lemma \ref{graphautolem}, $\sigma$ is either the identity or $(n{-}1 \hspace{3pt} n)$ so, seeking a contradiction, assume it is the former; that is, $\Sigma P_i = P_i$ for all $i$. Consider the morphism $P_n \to P_1$ given by left multiplication by $\alpha_1 \alpha_2 \alpha_3 \dots \alpha_{n-3} \alpha_n$. This gives rise to a distinguished triangle
\begin{align*}
P_n \to P_1 \to M \to P_n
\end{align*}
for some $M \in \proj \Pi(\mathbb{D}_n)$. Applying $\Hom(-, P_1)$ gives rise to the following exact sequence,
\begin{displaymath}
\begin{tikzcd}[column sep=small,>=stealth,matrix scale=1, transform shape, nodes={scale=1}] 
\Hom(P_1,P_1)  \arrow{r}[shift={(0pt,4pt)}]{\cdot \alpha_1 \alpha_2  \dots  \alpha_{n-3} \alpha_n} \arrow[equal]{d} & \Hom(P_n,P_1) \arrow{r}[shift={(0pt,4pt)}]{\beta} \arrow[equal]{d} & \Hom(M,P_1)  \arrow{r}[shift={(0pt,4pt)}]{\gamma} & \Hom(P_1,P_1) \arrow{r}[shift={(0pt,4pt)}]{\cdot \alpha_1 \alpha_2  \dots  \alpha_{n-3} \alpha_n} \arrow[equal]{d} & \Hom(P_n,P_1)  \arrow[equal]{d} \\
\hspace{10pt} \Bbbk e_1 \oplus \Bbbk p \hspace{10pt} & \Bbbk \alpha_1 \alpha_2  \dots  \alpha_{n-3} \alpha_n  & {} & \hspace{10pt} \Bbbk e_1 \oplus \Bbbk p \hspace{10pt}  & \Bbbk \alpha_1 \alpha_2  \dots  \alpha_{n-3} \alpha_n
\end{tikzcd}
\end{displaymath}
\noindent where here $p$ is some path. The left hand map is surjective, so $\beta$ is the zero map and therefore $\gamma$ is injective. The kernel of the right hand map is one-dimensional, and exactness tells us that $\gamma$ has rank 1. In particular we have $\dim \Hom(M,P_1) = 1$ and so $M$ is either $P_{n-1}$ or $P_n$ by Lemma \ref{homspacedims}. If we instead apply $\Hom(-,P_n)$, the resulting exact sequence is
\begin{displaymath}
\begin{tikzcd}[column sep=small,>=stealth,matrix scale=1, transform shape, nodes={scale=1}] 
\Hom(P_1,P_n)  \arrow{r}[shift={(0pt,4pt)}]{\cdot \alpha_1 \alpha_2  \dots  \alpha_{n-3} \alpha_n} \arrow[equal]{d} & \Hom(P_n,P_n) \arrow{r}[shift={(0pt,4pt)}]{\beta} \arrow[equal]{d} & \Hom(M,P_n) \arrow{r}[shift={(0pt,4pt)}]{\gamma} & \Hom(P_1,P_n) \arrow{r}[shift={(0pt,4pt)}]{\cdot \alpha_1 \alpha_2  \dots  \alpha_{n-3} \alpha_n} \arrow[equal]{d} & \Hom(P_n,P_n)  \arrow[equal]{d} \\
\Bbbk \overline{\alpha}_n \overline{\alpha}_{n-3} \dots \overline{\alpha}_2 \overline{\alpha}_1 & \Bbbk^{\tfrac{n-1}{2}} & {} & \Bbbk \overline{\alpha}_n \overline{\alpha}_{n-3} \dots \overline{\alpha}_2 \overline{\alpha}_1 & \Bbbk^{\tfrac{n-1}{2}}
\end{tikzcd}
\end{displaymath}
\noindent Since $n$ is odd, the shortest path from vertex $n$ to vertex 1 and back to vertex $n$ is zero in $\Pi(\mathbb{D}_n)$, so the first and the last maps both have rank zero. Therefore $\beta$ has full rank, forcing the kernel of $\gamma$ to have dimension $\tfrac{n-1}{2}$. Exactness also forces $\gamma$ to have rank 1, and therefore $\dim \Hom(M,P_n) = \tfrac{n-1}{2} + 1$. Now we have already seen that $M$ is either $P_{n-1}$ or $P_n$, but $\dim \Hom(P_{n-1},P_n) = \frac{n-1}{2} = \dim \Hom(P_n,P_n)$, so we have a contradiction. Therefore $\sigma$ is the unique graph automorphism of order 2.
\end{proof}

\begin{prop} \label{DnEvenProp}
Let $n \geqslant 6$ be even and let $\sigma$ be the graph automorphism of $\mathbb{D}_n$ induced by the translation functor $\Sigma$ on $\proj \Pi(\mathbb{D}_n)$. Then $\sigma$ is the identity.
\end{prop}
\begin{proof}
By Lemma \ref{graphautolem}, $\sigma$ is either the identity or $(n{-}1 \hspace{3pt} n)$ so, seeking a contradiction, assume it is the latter. Consider the morphism $P_n \to P_1$ given by left multiplication by $\alpha_1 \alpha_2 \dots \alpha_{n-3} \alpha_{n}$, and extend this to a distinguished triangle
\begin{align*}
P_n \to P_1 \to M \to P_{n-1}
\end{align*}
for some $M \in \proj \Pi(\mathbb{D}_n)$. If we apply $\Hom(-,P_1)$ we get the following exact sequence
\begin{displaymath}
\begin{tikzcd}[column sep=small,>=stealth,matrix scale=1, transform shape, nodes={scale=1}] 
\Hom(P_1,P_1)  \arrow{r}[shift={(0pt,4pt)}]{\cdot \alpha_1 \alpha_2 \dots \alpha_{n-3} \alpha_{n-1}} \arrow[equal]{d} & \Hom(P_{n-1},P_1)  \arrow{r}[shift={(0pt,4pt)}]{\beta} \arrow[equal]{d} & \Hom(M,P_1)  \arrow{r}[shift={(0pt,4pt)}]{\gamma} & \Hom(P_1,P_1) \arrow{r}[shift={(0pt,4pt)}]{\cdot \alpha_1 \alpha_2 \dots \alpha_{n-3} \alpha_{n}} \arrow[equal]{d} & \Hom(P_n,P_1)  \arrow[equal]{d} \\
\hspace{10pt} \Bbbk e_1 \oplus \Bbbk p & \Bbbk \alpha_1 \alpha_2 \dots \alpha_{n-3} \alpha_{n-1} & {} & \hspace{10pt} \Bbbk e_1 \oplus \Bbbk p \hspace{10pt} & \Bbbk \alpha_1 \alpha_2 \dots \alpha_{n-3} \alpha_{n}
\end{tikzcd}
\end{displaymath}
\noindent where $p$ is some path. The left hand map surjects, so $\beta = 0$ and therefore $\gamma$ injects. The right hand map has a one-dimensional kernel, so $\gamma$ has rank 1. It follows that $\dim \Hom(M,P_1) = 1$, which implies that $M$ is either $P_{n-1}$ or $P_n$. If we instead apply $\Hom(-,P_n)$ we get
\begin{displaymath}
\begin{tikzcd}[column sep=small,>=stealth,matrix scale=1, transform shape, nodes={scale=1}] 
\Hom(P_1,P_n)  \arrow{r}[shift={(0pt,4pt)}]{\cdot \alpha_1 \alpha_2  \dots \alpha_{n-3} \alpha_{n-1}} \arrow[equal]{d} & \Hom(P_{n-1},P_n) \arrow{r}[shift={(0pt,4pt)}]{\theta} \arrow[equal]{d} & \Hom(M,P_n) \arrow{r}[shift={(0pt,4pt)}]{\eta} & \Hom(P_1,P_n) \arrow{r}[shift={(0pt,4pt)}]{\cdot \alpha_1 \alpha_2  \dots \alpha_{n-3} \alpha_{n}} \arrow[equal]{d} & \Hom(P_n,P_n)  \arrow[equal]{d} \\
\Bbbk \overline{\alpha}_{n} \overline{\alpha}_{n-3} \dots \overline{\alpha}_{2} \overline{\alpha}_{1} & \Bbbk^{\tfrac{n}{2}-1} & {} & \Bbbk \overline{\alpha}_{n} \overline{\alpha}_{n-3} \dots \overline{\alpha}_{2} \overline{\alpha}_{1}  & \Bbbk^{\tfrac{n}{2}}
\end{tikzcd}
\end{displaymath}
\noindent Since $n$ is even, the shortest path from vertex $n$ to vertex $1$ and then to vertex $n{-}1$ is zero in $\Pi(\mathbb{D}_n)$, while the shortest path from vertex $n$ to vertex $1$ and back to vertex $n$ is nonzero. It follows that the left hand map is the zero map, while the right hand map has rank 1. Therefore $\theta$ has full rank which implies that the kernel of $\eta$ has dimension $\tfrac{n}{2}-1$. Moreover, the right hand map is injective, so that $\eta$ has rank 0 and so $\dim \Hom(M,P_n) = \frac{n}{2}-1$. Combining this with our earlier restriction on $M$, this forces $M = P_{n-1}$, and our distinguished triangle is therefore 
\begin{align*}
P_n \to P_1 \to P_{n-1} \to P_{n-1}.
\end{align*}
Since $\Hom(P_1,P_{n-1})$ is spanned by $\overline{\alpha}_{n-1} \overline{\alpha}_{n-3} \dots \overline{\alpha}_2 \overline{\alpha}_1$ and $\gamma$ is not the zero map, we can assume that the map $P_1 \to P_{n-1}$ in this triangle is given by left multiplication by (a scalar multiple of) $\overline{\alpha}_{n-1} \overline{\alpha}_{n-3} \dots \overline{\alpha}_2 \overline{\alpha}_1$. Moreover,  $\Hom(P_{n-1},P_{n-1}) = \sspan \{ e_{n-1}, p_2, \dots , p_{n/2} \}$ where the $p_i$ are paths of length $\geqslant 4$. Since $\theta$ is not the zero map and the composition $P_1 \to P_{n-1} \to P_{n-1}$ must be zero, the map $P_{n-1} \to P_{n-1}$ in this triangle lies in $\sspan \{ p_2, \dots , p_{n/2} \}$. But then $\theta : \Hom(P_{n-1}, P_n) \to \Hom(P_{n-1}, P_n)$ maps the longest path in $\Hom(P_{n-1}, P_n)$ to zero, contradicting the fact that $\theta$ has trivial kernel. It follows that $\sigma$ is not a two-cycle.
\end{proof}


\bibliographystyle{amsalpha}
\bibliography{thesisbib}


\end{document}